\documentclass[11pt]{article}
\usepackage{}
\usepackage{enumerate}
\usepackage{mathbbold}
\usepackage{bbm}
\usepackage{amsfonts}
\usepackage[T1]{fontenc}
\usepackage{latexsym,amssymb,amsmath,amsfonts,amsthm}
\usepackage{graphicx, color}
\usepackage{subfigure}
\usepackage{epsf}
\usepackage{epstopdf}
\usepackage{amssymb}
\usepackage{mathrsfs}
\usepackage[linesnumbered, ruled]{algorithm2e}
\usepackage{multirow}
\usepackage{appendix}

\SetKwRepeat{Do}{do}{while}

\renewcommand{\theequation}{\arabic{section}.\arabic{equation}}

\newcommand{\R}{{\mathbb R}}

\newcommand{\C}{{\mathbb C}}

\newcommand{\be}{\begin{eqnarray}}
\newcommand{\ben}{\begin{eqnarray*}}
\newcommand{\en}{\end{eqnarray}}
\newcommand{\enn}{\end{eqnarray*}}
\newcommand{\ba}{\backslash}
\newcommand{\pa}{\partial}

\newcommand{\ov}{\overline}

\newcommand{\G}{\Gamma}

\newcommand{\wi}{\widetilde}

\newcommand{\hx}{\hat{x}}
\SetKwRepeat{Do}{do}{while} 

\newtheorem{theorem}{Theorem}[section]

\newtheorem{corollary}[theorem]{Corollary}

\begin{document}
\title{\bf Uniqueness and modified Newton method for cracks from the far field patterns with a fixed incident direction}
\author{
Jialei Li\thanks{School of Mathematical Sciences, University of Chinese Academy of Sciences, Beijing 100049, China, and Academy of Mathematics and Systems Science,
Chinese Academy of Sciences, Beijing 100190, China. Email: lijialei21@mails.ucas.ac.cn}
,\and
Xiaodong Liu\thanks{Academy of Mathematics and Systems Science,
Chinese Academy of Sciences, Beijing 100190, China. Email: xdliu@amt.ac.cn}
}
\date{}
\maketitle

\begin{abstract}
We consider the inverse cracks scattering problems from the far field patterns with a fixed incident direction. We firstly show that the sound-soft cracks can be uniquely determined by the multi-frequency far field patterns with a fixed incident direction. The proof is based on a low frequency asymptotic analysis of the scattered field. One important feature of the uniqueness result is that the background can even be an unknown inhomogeneous medium.
A modified Newton method is then proposed for the numerical reconstruction of the shapes and locations of the cracks.
Compared to the classical Newton method, the modified Newton method relaxes the dependence of a good initial guess and can be applied for multiple cracks.
Numerical examples in two dimensions are presented to demonstrate the feasibility and effectiveness of the modified Newton method.
In particular, the quality of the reconstructions can be greatly improved if we use the measurements properly with two frequencies or two incident directions.

\vspace{.2in}
{\bf Keywords:} Inverse scattering problem, cracks identification, multi-frequency data, Newton method.

\vspace{.2in} {\bf AMS subject classifications:} 35R30, 35P25

\end{abstract}

\section{Introduction}
\setcounter{equation}{0}
The inverse crack scattering problem plays an important role in detecting cracks hidden in diverse structures such as buildings, bridges or other concrete structure. The unknown object is a collection $\Sigma = \{\Gamma_1, \Gamma_2, \cdots, \Gamma_n\}$ of $n$ cracks.  Every crack $\Gamma_j$ is an one-dimensional connected sub-manifold of $\mathbb{R}^2$ with
the parametric representation
\ben
z_j: (-1,1)\rightarrow\mathbb{R}^2,
\enn
which is an injective and twice continuously differentiable function such that $z_j'(s)\neq 0$ for all $s\in(-1,1)$, $j=1,2,\cdots, n$.
Here and throughout the paper, we suppose any two cracks have positive distance such that the complement $\mathbb{R}^2\backslash \overline{\Sigma}$ is connected.
Given the plane wave $u^i=e^{ikx\cdot d}$ with the frequency $k>0$ and the incident direction $d\in S^1:=\{x\in \mathbb{R}^2\,|\,|x|=1\}$, the scattering of $u^i$ by sound-soft cracks is formulated as
\begin{equation}
\label{eq-scattering}
    \left\{
\begin{aligned}
    &\Delta u + k^2 u = 0 &\hbox{in }\mathbb{R}^2\backslash \overline{\Sigma} \\
    &u = 0                 &\hbox{on } \Sigma \\
    & \lim_{r\to\infty} \sqrt{r}\left(\frac{\partial u^s }{\partial r} - ik u^s \right)=0
\end{aligned}
\right.,
\end{equation}
where $u$ and $u^s:= u- u^i$ are called the total field and the scattered field, respectively.
The limit in \eqref{eq-scattering} is the well known Sommerfeld radiating condition, which holds uniformly for all directions $\hx:=x/|x|\in S^1$.
From this, the scattered field $u^s$ has an asymptotic behavior of the form
\ben
u^s(x;k) =\frac{e^{ik|x|}}{\sqrt{|x|}}\left\{
u^\infty(\hat{x};k)+O\left(\frac{1}{|x|}\right)
\right\},\quad \mbox{as } |x|\to \infty,
\enn
uniformly in all directions $\hat{x}\in S^1$. The function $u^\infty$, defined on the unit circle $S^1$, is known as the the far field pattern of $u^s$.
The inverse crack problem is to recover the cracks $\Sigma$ from a knowledge of the scattered field $u^s$ or the corresponding far field pattern $u^\infty$.

Since the pioneer work \cite{Kress95} by Kress for a sound soft crack in 1995, numerous techniques and numerical methods have been proposed to solve this problem. The inverse cracks scattering problem is extended to various boundary conditions, for example, the sound hard case \cite{Monch97}, the impedance case \cite{KressLee03}, and the mixed-type case \cite{CacoColtMonk01,KruLiuSini08,Yan08}. Furthermore, the results are also applicable to other problem settings, such as the elastic wave equation \cite{Kress96,NakaUhlmWang03}, inverse conductivity problem \cite{AlvaDIM09}, electromagnetic wave equation \cite{Park14}, and others discussed in \cite{GaoDM18,HauptmannIIS19,KressVinto08,Valtul04}.

The numerical methods for cracks can be categorized into non-iterative schemes and iterative schemes. Non-iterative schemes, such as the factorization method \cite{KirschRitter00, XuMH24}, linear sampling method \cite{CakoniColton03}, direct sampling method \cite{Park18}, MUltiple SIgnal Classification (MUSIC) type method \cite{AmmariKLP10,AmmariGKP11}, offer fast detection without the need of an initial guess, although the imaging resolution is limited. On the other hand, iterative schemes, such as Newton's method \cite{Kress95,Kress96,Monch97}, hybrid method \cite{KressSerra05}, two-step method \cite{Lee10}, and level-set method \cite{AlvaDIM09}, can give a better reconstruction with the price of a good initial guess.

Following the idea of Isakov \cite{Isakov90}, it is shown in \cite{CacoColtMonk01,Kress95,KressLee03,KruLiuSini08,Monch97,Yan08} that the cracks can be uniquely determined from the far field patterns for all incident directions at a fixed frequency.
The main technique in the proof is to construct a sequence of solutions with a singularity moving towards a boundary point, which belongs to one crack but does not belong to the other. To do so, one has to use the far field patterns for all incident directions. Unfortunately, this is not practical since it is too expensive to retrieve so much data.
In contrast, we are interested in deriving uniqueness results from the far field patterns with a fixed incident direction, in which Isakov's technique is not applicable. If it is assumed that the sound soft cracks are line segments, then a single incident plane wave is sufficient to uniquely determine the cracks as established by Liu and Zou \cite{LiuZou06}. The proof is based on the reflection principle for the solution of the Helmholtz equation. We refer to \cite{HuLiu14} for the uniqueness if the incident field is a point source.
However, the uniqueness of general cracks from the far field patterns with a single incident direction is still a challenging open problem.

The first contribution of this paper is, for the first time, to establish a uniqueness result for general sound soft cracks from the multi-frequency far field patterns with a fixed incident direction. The proof is based on a low frequency expansion of the scattered fields,  which is motivated by the techniques in \cite{LiuLiu17} for obstacles. One important feature of this uniqueness result is that the background can even be an unknown inhomogeneous medium.  In \cite{Kress99}, Kress studied the low frequency asymptotics for the scattered field via a single-layer integral equation approach. Consequently, Kress showed that the solutions to the Dirichlet problem for the Helmholtz equation converge to a solution of the Dirichlet problem for the Laplace equation as the frequency tends to zero provided the boundary values converge. Following \cite{Kress99,LiuLiu17}, we
give an explicit low frequency expansion of the total field, which is then used to prove the uniqueness of the cracks. 
Note that, in the uniqueness proof of obstalces\cite{LiuLiu17}, an important result is that two distinct obstacles $D_1$ and $D_2$ always give rise to a nonempty domain $D_1\backslash \overline{D_2}$ (or similarly $D_2\backslash \overline{D_1}$). Unfortunately, such a domain can not be generated by two distinct cracks. We prove for the first time that uniqueness of the inverse problem  remains valid for cracks. The key step is an argument of the maximal principle of certain term of the low frequency expansion.

%

The second contribution of this paper is a modified Newton method that updates all coordinates. In this way, compared with the Newton method proposed in \cite{Kress95} and the hybrid method in \cite{KressSerra05}, our method is more flexible and can avoid the self-intersecting curves during the iterations. Numerical experiments verify the effectiveness and robustness in various configurations. To improve the reconstruction quality and reduce the reliance on the initial guess, we introduce a proper choice of measurements with two frequencies or two incident directions. In recent years, the inverse scattering problem from multi-frequency data has been extensively investigated for various scattering objects, including source \cite{BaoLRX15,LiLiu23,LiLiuShi23}, medium \cite{BaoTriki10}, and obstacle \cite{LiuLiu17}. For inverse cracks scattering, multi-frequency data is only used in some non-iterative numerical methods \cite{AmmariGKP11,Park14}.


The paper is structured as follows. In Section 2, we analyze the low frequency behavior of the total field and establish the uniqueness of the inverse problem. Section 3 is devoted to the modified Newton method. In Section 4, numerical examples are presented to demonstrate the feasibility and robustness of the modified Newton method. Finally in the Appendix, we discuss about the uniqueness of the impedance cracks, including a counterexample of the cracks with the impedance boundary conditions, and show the uniqueness of sound soft cracks in an unknown inhomogeneous medium.

\section{Uniqueness for the inverse problem with multi-frequency data of one incident direction}
\setcounter{equation}{0}
In this section, we prove the uniqueness of the inverse problem as stated in the following Theorem.

\begin{theorem}
\label{thm-unique}
Assume $\Sigma$ and $\hat{\Sigma}$ are two collections of sound soft cracks which produce the same far field patterns for incident fields $u^i = e^{ik x\cdot d}$ with all $k \in \{k_j>0|j = 1,2,\cdots, k_j \to 0\}$ and a fixed $d\in S^1$. Then $\Sigma=\hat{\Sigma}$.
\end{theorem}

To do so, we recall an integral representation of the scattered field $u^s$ and analyze its low frequency asymptotic behavior. Note that the more complex inhomogeneous background case will be presented in the Appendix.

For any $y=z_j(s)\in \Gamma_j, s\in (-1,1)$, by the bijectivity of the cosine function from $(0,\pi)$ to $(-1, 1)$, we may take $t\in (0, \pi)$ such that $s=\cos t$ and therefore $y=z_j(\cos t), \, t\in (0, \pi)$. Consequently, we define a weighted functional space $\mathcal{X}\subset L^p(\Sigma),\,1<p<4/3$ by
\begin{equation*}
\mathcal{X} := \left\{
\phi(y)\in L^p(\Sigma)\,\Big|\,\phi(z_j(\cos t))= \frac{\psi_j(t)}{|z_j'(\cos t)| \sin t}\, \mbox{for}\, \, \psi_j(t)\in L^2(0,\pi),\, 1\leq j\leq n\right\}.
\end{equation*}
Such a space is naturally derived from the cosine transformation, we refer to \cite{CakoniColton03,Kress95} for more details.
Given a density function $\phi\in\mathcal{X}$, we define the single layer potential $\mathcal{S}_k:\mathcal{X}\to H^1_{loc}(\mathbb{R}^2\backslash \overline{\Sigma})$ by
\ben
( \mathcal{S}_k\phi) (x) := \int_\Sigma \Phi_k(x,y) \phi(y) {\rm d}y,\quad x\in \mathbb{R}^2\backslash\overline{\Sigma},
\enn
where
\ben
\Phi_k(x,y) = \frac{i}{4}H^{(1)}_0(k|x-y|),\quad x\in\R^2\ba\{y\},
\enn
is the fundamental solution to the Helmholtz equation in terms  of the Hankel function $H^{(1)}_0$ of the first kind and of order zero.
The restriction of $\mathcal{S}_k$ to the cracks $\Sigma$ is denoted by $\mathrm{S}_k:\mathcal{X}\to H^{1}(\Sigma)$, which is well defined and admits a bounded inverse from $H^{1}(\Sigma)$ to $\mathcal{X}$ \cite{Kress99}.
To study the low frequency behavior of the scattered field, following Kress \cite{Kress99}, we introduce a boundary operator $W: \mathcal{X}\rightarrow \mathcal{X}$ given by
\ben
(W\phi)(x) = \phi(x) - \frac{1}{|\Sigma|}\int_\Sigma  \phi(y) {\rm d}y,  \quad x\in \Sigma.
\enn
Here, $|\Sigma| = \sum_{j=1}^n |\Gamma_j|$ is the total length of all cracks. Then we have the following representation of the total field $u$.
\begin{theorem}
\label{thm-int-representation2}
Let $u\in H^1_{loc}(\mathbb{R}^2\backslash\overline{\Sigma})$ be a solution to \eqref{eq-scattering} for incident field $u^i$. Assume that $u^i|_\Sigma\in H^1(\Sigma)$. Then the total field $u\in H^1_{loc}(\mathbb{R}^2\backslash\overline{\Sigma})$ has the following form
\be\label{eq-sys2-1}
u = u^i + \mathcal{S}_k\left(W-\frac{2\pi}{\ln k}I\right)\phi \quad \mbox{in } \mathbb{R}^2\backslash\overline{\Sigma},
\en
where $\phi \in \mathcal{X}$ is determined by the boundary integral equation
\be
\label{eq-sys2-2}
\mathrm{S}_k\left(W-\frac{2\pi}{\ln k}I\right)\phi=-u^i\quad \mbox{on } \Sigma.
\en
\end{theorem}

Theorem \ref{thm-int-representation2} has been proved by Kress in \cite{Kress99} for a single crack, i.e., $n=1$. The extension from $n=1$ to $n>1$ is straightforward, and thus we omit the proof here.

Recall the fundamental solution of the Laplace equation given by
\ben
\Phi_0(x,y)= -\frac{1}{2\pi} \ln |x-y|, \quad x\in \R^2\ba\{y\}.
\enn
Denote by $\mathcal{S}_0$ and $\mathrm{S}_0$, respectively, the single layer potential $\mathcal{S}$ and boundary operator $\mathrm{S}$ for the limiting case $k=0$ with $\Phi_k$ replaced by $\Phi_0$.
Straightforward calculations show that
\begin{equation}
\label{eq-asym-Phik}
\Phi_k(x,y) = \Phi_0(x,y)  -\frac{1}{2\pi} \ln{k} + C  + O(k^2\ln{k}),\quad k\rightarrow 0,
\end{equation}
where $C = \ln{2}/2\pi - \gamma/2\pi + i/4$ with $\gamma=0.5772...$ denoting the Euler's constant. This implies immediately that $\mathrm{S}_k\phi$ does not converge to
$\mathrm{S}_0\phi$ as $k\rightarrow 0$.
We define
\ben
 L\phi  := \int_\Sigma  \phi(y) {\rm d}y.
\enn
Using \eqref{eq-asym-Phik}, with the help of the obvious fact that $LW=0$, we have
\ben
\mathrm{S}\left(W-\frac{2\pi}{\ln k}I\right)\phi
&=&(\mathrm{S}_0-\frac{1}{2\pi} \ln{k}L + CL+ O(k^2\ln{k}))\left(W-\frac{2\pi}{\ln k}I\right)\phi\\
&=&\left[A -\frac{2\pi}{\ln k}(\mathrm{S}_0+CL) + O(k^2\ln k)\right] \phi, \quad k\rightarrow 0,
\enn
where $A:=\mathrm{S}_0W+L$ is well defined from $\mathcal{X}$ to $H^1(\Sigma)$ and has a bounded inverse (see e.g., Theorem 3.2 in \cite{Kress99}).
Furthermore, for the incident plane wave $u^i=e^{ikx\cdot d}$ with incident direction $d\in S^1$, we have from the boundary condition \eqref{eq-sys2-2} that
\ben
\phi
&=&-\left[\mathrm{S}\left(W-\frac{2\pi}{\ln k}I\right)\right]^{-1}u^i\\
&=&-\left[A -\frac{2\pi}{\ln k}(\mathrm{S}_0+CL) + O(k^2\ln k)\right]^{-1}(1 + O(k))\\
&=&-\left[A^{-1} + \frac{2\pi}{\ln k}  A^{-1} (\mathrm{S}_0+CL)A^{-1}  + O\left(\frac{1}{(\ln k)^{2}}\right)\right] (1 + O(k))\\
&=&- A^{-1}(1) - \frac{2\pi}{\ln k}  A^{-1} (\mathrm{S}_0+CL)A^{-1}(1)  + O\left(\frac{1}{(\ln k)^{2}}\right),\quad k\rightarrow 0.
\enn
Inserting this into \eqref{eq-sys2-1}, with the help of the obvious fact ${A}^{-1}(1) = 1/|\Sigma|$ and $L{A}^{-1}(1)=1$, we have the following theorem on the low frequency expansion of the total fields.

\begin{theorem}\label{u_lowfrequency}
Let $u^i=e^{ik x\cdot d}$ be the incident field with some fixed direction $d\in S^1$. Then the solution $u$ of problem \eqref{eq-scattering} has the following low frequency expansion
\be
\label{eq-asym-us}
u(\cdot, k)=- \frac{2\pi}{|\Sigma|\ln{k}} \left[(\mathcal{S}_0 {W} +{L}){A}^{-1}\mathrm{S}_0(1)-\mathcal{S}_0(1)\right]  + O\left(\frac{1}{(\ln k)^{2}}\right)\quad k\rightarrow0
\en
in $\mathbb{R}^2\backslash\overline{\Sigma}$.
\end{theorem}


%

Now we are ready to show Theorem \ref{thm-unique}.
\begin{proof}[Proof of Theorem \ref{thm-unique}]
Assume that $\Sigma\neq \hat{\Sigma}$. Let $G$ be the unbounded connected component of $\mathbb{R}^2\backslash(\overline{\Sigma}\cup\overline{\hat{\Sigma}})$.
By the well known Rellich's lemma and the analytic continuation, the total waves coincide in $G$.

We define
\begin{equation}
\label{eq-uniq-u}
v(x) = \frac{1}{|\Sigma|}[(\mathcal{S}_0{W}+{L}){A}^{-1}\mathrm{S}_0 (1) - \mathcal{S}_0 (1)],\quad x\in \R^2\ba\Sigma.
\end{equation}
Then $v$ is harmonic in $\mathbb{R}^2\backslash\overline{\Sigma}$ and $v(x) = 0$ on $\Sigma$.
Similarly we define $\hat{v}$ by
\ben
\hat{v}(x) = \frac{1}{|\hat{\Sigma}|}[(
\hat{\mathcal{S}}_0\hat{W}+\hat{L})\hat{A}^{-1}\hat{\mathrm{S}}_0 (1) - \hat{\mathcal{S}}_0 (1)], \quad x\in \R^2\ba\hat{\Sigma}.
\enn
Then $\hat{v}$ is harmonic in $\mathbb{R}^2\backslash\overline{\hat{\Sigma}}$ and $v(x) = 0$ on $\hat{\Sigma}$.
An application of Theorem \ref{u_lowfrequency} yields
\be\label{v=hatv}
\left\{
\begin{aligned}
v=\hat{v} &&\mbox{in } G,\\
v=\hat{v}=0 &&\mbox{on } \partial G.
\end{aligned}
\right.
\en

We claim that $G=\mathbb{R}^2\backslash (\overline{\Sigma}\cup\overline{\hat{\Sigma}})$. Otherwise, let $\mathring{G^c}$ be the interior of $G^c:=\R^2\ba G$, we define $D:=\mathring{G^c}\backslash \overline{\Sigma}$. Then $D$ is a bounded domain such that $\partial D \subset \partial G\cup\overline{\Sigma}$.
Hence we find that $v$ solves the following interior Laplace problem
\ben
\left\{
\begin{aligned}
\Delta v &= 0 &\mbox{in } D,\\
 v&=0 &\mbox{on } \partial D.
\end{aligned}
\right.
\enn
Consequently, we have $v=0$ in $D$ by the maximum principle. Moreover, the unique continuation yields $v=0$ in $\mathbb{R}^2$. This contradicts to \eqref{eq-uniq-u} since by straightforward calculation we have
\ben
v(x) = \frac{1}{2\pi} \ln{|x|} + O(1),\quad |x|\to \infty,
\enn
uniformly in all directions $\hx$. Therefore, $G=\mathbb{R}^2\backslash (\overline{\Sigma}\cup\overline{\hat{\Sigma}})$.

Assume that $\Gamma:=\hat{\Sigma}\backslash \overline{\Sigma} \neq \emptyset$. Then, from \eqref{v=hatv}, we have $v=\hat{v}=0$ on $\Gamma\subset\pa G$. This yields a contradiction since $\Gamma\subset\mathbb{R}^2\backslash\overline{\Sigma}$ and $v>0$ in $\mathbb{R}^2\backslash\overline{\Sigma}$ by maximum principle. Therefore $\hat{\Sigma}\subset \Sigma$ and vice versa. The proof is finished.
\end{proof}

\textbf{Remark}: Noting the fact that the far field pattern depends analytically on the frequency $k$, the uniqueness theorem holds if we have the far field patterns for all frequencies contained in some interval for $0<k_-<k<k_+<\infty$.

\section{Modified Newton Method}
\setcounter{equation}{0}

We introduce a modified Newton method to reconstruct the sound soft cracks. Compared with the Newton method proposed by Kress in \cite{Kress95}, our novel method relaxes the assumption of the initial guess and can be applied to multiple cracks. Different to \cite{KressSerra05}, our techniques avoid the case of self-intersecting curves.

For the solution to the exterior boundary problem \eqref{eq-scattering} with $k>0$, we look for the scattered field in the form $u^s=\mathcal {S}_k\phi$ with $\phi\in\mathcal {X}$ obtained by the boundary condition $\mathrm{S}_k\phi=-u^i$. Consequently, we may write the scattered field in the form
\ben
u^s  = \mathcal{S}_k \mathrm{S}_k^{-1} (-u^i)\quad \mbox{in}\,\R^2\ba\Sigma.
\enn


We defines an operator
\begin{equation}
\label{eq-F}
\mathbf{F}: \Sigma \mapsto u^\infty ,
\end{equation}
which maps the cracks $\Sigma$ to  the far field $u^\infty $. The inverse problem is now reduced to solve the nonlinear and ill-posed equation
\ben
\mathbf{F}(\Sigma) = u^\infty,
\enn
which has the linearized form
\begin{equation}
\label{eq-NewtonEq}
\mathbf{F}(\Sigma) + \mathbf{F}'(h;\Sigma)= u^\infty .
\end{equation}
Here $\mathbf{F}'(h;\Sigma)$ is the Fr\'{e}chet derivatives of $\mathbf{F}$ with respect to $\Sigma$ at the direction of $h$.

\subsection{Fr\'{e}chet derivatives}

\begin{theorem}
\label{thm-Frechet}
The operator $\mathbf{F}: \Pi_{i=1}^n C^3(-1,1)\to L^2(S^1)$ in \eqref{eq-F} is  Fr\'{e}chet differentiable. Moreover, $\mathbf{F}'(h;\Sigma)(\hat{x})=v^\infty(\hat{x})$, where $v\in H^1_{loc} (\mathbb{R}^2\backslash \overline{\Sigma})$ is the radiating solution of the Helmholtz equation
\begin{equation}
\label{eq-Helmholtz4v}
    \left\{
\begin{aligned}
    &\Delta v + k^2 v = 0 &\hbox{in }\mathbb{R}^2\backslash \overline{\Sigma}, \\
    &v = -\nu \cdot  Z^{-1}\tilde{h} \frac{\partial u_{\pm}}{\partial \nu}  &\hbox{on } \Sigma
\end{aligned}
\right.
\end{equation}
such that $v-v_0$ is continuous in $\mathbb{R}^2$ with
\begin{equation}
\label{eq-v0}
v_0(x) := \int_\Sigma \nabla_x \Phi_k(x,\cdot)\cdot \tilde{h} \Big[\frac{\partial u_{+}}{\partial \nu}-\frac{\partial u_{-}}{\partial \nu} \Big] {\rm d}s, \quad x\in \mathbb{R}^2\backslash \overline{\Sigma}.
\end{equation}
Here, $u$ is the solution of the scattering problem \eqref{eq-scattering} and $\tilde{h_i}(t) = h_i(\cos t)$ for $t\in(0,\pi)$.
\end{theorem}
The special case with one crack has been established by Kress in \cite{Kress95}. For completeness, we generalize this result to $n$ cracks.
\begin{proof}
To show $\mathbf{F}:\Sigma \mapsto u^\infty $ is Fr\'{e}chet differentiable, we prove first that $F:\Sigma \mapsto u^s= \mathcal{S}_k \mathrm{S}_k^{-1} (-u^i)$ is Fr\'{e}chet differentiable. Then $\mathbf{F}'(h;\Sigma)$ is exactly the far field pattern corresponding to $F'(h;\Sigma)$.

We rewrite $\mathcal{S}_k$ as
\ben
(\mathcal{S}_k (\psi;\Sigma))(x) = \sum_{i=1}^n \int_0^\pi \Phi_k(x, z_i(\cos \tau)) \psi_i(\tau) {\rm d} \tau, \quad x\in D,
\enn
where $\psi:=(\psi_1,\psi_2,\cdots, \psi_n)^T\in\Pi_{i=1}^n L^2(0,\pi)$, $z=(z_j)_{1\leq j\leq n}$. Here, $D\subset \mathbb{R}^2$ is a bounded domain that has a positive distance to ${\Sigma}$. For such $D$, any small deviation $\Sigma+h$ of $\Sigma$ should also have a positive distance to $D$, which makes sure that we can calculate the Fr\'{e}chet derivatives of all $h$.
Correspondingly, we rewrite the restriction $\mathrm{S}_k$ on $\Sigma$ as
\ben
\mathrm{S}_k =
\begin{pmatrix}
\mathrm{S}_k^{1,1} & \mathrm{S}_k^{1,2}   & \cdots & \mathrm{S}_k^{1,n} \\
\mathrm{S}_k^{2,1} & \mathrm{S}_k^{2,2}   & \cdots & \mathrm{S}_k^{2,n} \\
\vdots & \vdots   & \ddots & \vdots \\
\mathrm{S}_k^{n,1} & \mathrm{S}_k^{n,2}   & \cdots & \mathrm{S}_k^{n,n}
\end{pmatrix},
\enn
where $\mathrm{S}_k^{j,i}$ is given by
\ben
\mathrm{S}_k^{j,i} (\psi_i;\Sigma) = \int_0^\pi \Phi_k(z_j(\cos t), z_i(\cos \tau)) \psi_i(\tau) {\rm d} \tau,\quad   t\in (0,\pi).
\enn
Then the scattered field is formulated as $u^s=F(\Sigma)  = \mathcal{S}_k \mathrm{S}_k^{-1} Z(-u^i)$, where
\ben
Z(f;\Sigma) =
\Big(
f(z_j(\cos t_j))
\Big)_{1\leq j\leq n}
\enn
transforms $f$ onto the  $\Pi_{i=1}^n H^1(0,\pi)$. Now we prove that $\mathcal{S}_k, \mathrm{S}_k^{-1},Z$ are all Fr\'{e}chet differentiable.
\begin{itemize}
  \item Clearly $Z$ is Fr\'{e}chet differentiable and
\begin{equation}
\label{eq-fre-Z}
Z'(f;\Sigma,h) =
\Big(
h_j(\cos t_j)\cdot \nabla f(z_j(\cos t_j))
\Big)_{1\leq j\leq n}.
\end{equation}
  \item For $\mathrm{S}_k$, it's proved in \cite{Kress95} that the diagonal term $\mathrm{S}_k^{i,i}$ is Fr\'{e}chet differentiable if $z_i\in C^3((-1,1),\R^2),  i=1,2,\cdots,n$. Moreover, the Fr\'{e}chet derivative of $\mathrm{S}_k^{i,i}$ is
\begin{equation}
\label{eq-fre-Sii}
{\mathrm{S}_k^{i,i}}'(\psi_i; z,h) (t) =\int_0^{\pi} H'(t,\tau;z_i,h_i) \psi_i(\tau){\rm d}\tau,\quad t\in(0,\pi),
\end{equation}
with the kernel
\begin{equation*}
H'(t,\tau;z_i,h_i) = \nabla_x \Phi_k(z_i(\cos t), z_i(\cos \tau))\cdot [h_i(\cos t) - h_i(\cos \tau)],\quad t\neq \tau.
\end{equation*}
It remains to prove that $\mathrm{S}_k^{i,j}$ is Fr\'{e}chet differentiable when $i\neq j$. Since $\Gamma_i$ and $\Gamma_j$ have positive distance, the kernel has no singularity and
\ben
\mathrm{S}_{k}^{i,j} (\psi_j; \Sigma+h) =\mathrm{S}_{k}^{i,j} (\psi_j;\Sigma) + {\mathrm{S}_{k}^{i,j}}'( \psi_j;\Sigma, h) + O(\|h_j\|_{C^2}^2),
\enn
where
\be\label{eq-fre-Sij}
&&{\mathrm{S}^{i,j}_{k}}'( \psi_j,\Sigma,  h)(t)\cr
& =&   \int_0^\pi [h_i(\cos t) -h_j(\cos \tau)] \cdot\nabla_x \Phi_k(z_i(\cos t), z_j(\cos \tau)) \psi_j(\tau) {\rm d} \tau,\quad t\in(0,\pi).\quad
\en
Therefore, $\mathrm{S}_k$ is Fr\'{e}chet differentiable with $\mathrm{S}_k'=({\mathrm{S}_k^{i,j}}')_{1\leq i,j\leq n}$.
Furthermore, $\mathrm{S}_k^{-1}$ is also Fr\'{e}chet differentiable with
\be\label{eq-fre-Sinv}
{\mathrm{S}_k^{-1}}'(\cdot;\Sigma,h)  = - \mathrm{S}_k^{-1}\mathrm{S}_k'(\cdot;\Sigma,h) \mathrm{S}_k^{-1}.
\en
  \item For $\mathcal{S}_k$, straightforward calculations show that
\ben
\mathcal{S}_{k} (\psi;\Sigma+h) = \mathcal{S}_{k} (\psi;\Sigma) + \mathcal{S}_{k}'( \psi; \Sigma, h) + O(\|h\|_{C^3}^2),
\enn
where
\be\label{eq-fre-S}
\mathcal{S}_{k}'( \psi;\Sigma, h) =  -\sum_{j=1}^n \int_0^\pi h_j(\cos \tau) \cdot\nabla_x \Phi_k(x, z_j(\cos \tau)) \psi_j(\tau) {\rm d} \tau,\quad x\in D.
\en
Thus $\mathcal{S}_k$ is Fr\'{e}chet differentiable. By \eqref{eq-fre-S}, $\mathcal{S}_{k}'( \psi;\Sigma, h)$ can be analytically extended to $\mathbb{R}^2\backslash\overline{\Sigma}$.  Moreover,
\begin{equation*}
\mathcal{S}_{k}'( \psi;\Sigma, h)|_{\G_i^\pm} =-\sum_{j=1}^n \int_0^\pi h_j(\cos \tau) \cdot\nabla_x \Phi_k(z_i(\cos(t)), z_j(\cos \tau)) \psi_j(\tau) {\rm d} \tau  \pm \frac{1}{2} \nu\cdot Z^{-1}\tilde{h}_i \phi_i.
\end{equation*}
Here $\tilde{h_i}(t) = h_i(\cos t)$ for $t\in(0,\pi)$.
Using the jump relation $\phi=\frac{\partial u_-}{\partial \nu}-\frac{\partial u_+}{\partial \nu}$, in terms of \eqref{eq-fre-Sii}-\eqref{eq-fre-Sij}, we have
\begin{equation}
\label{eq-SonT}
\mathcal{S}_{k}'( \psi;\Sigma, h)|_{\G_i^\pm} = -Z^{-1}\tilde{h}_i\cdot \nabla u^s _\pm + Z^{-1}\mathrm{S}_k'(\psi;\Sigma,h)|_{\G_i}
\end{equation}
on $\Gamma_i=\{z_i(t)\}\subset \Sigma$.
\end{itemize}

Now we are ready to discuss the Fr\'{e}chet derivative of $F$.
By the chain rule, with the help of \eqref{eq-fre-Z}, \eqref{eq-fre-Sinv} and \eqref{eq-fre-S}, we derive that the $F$ is Fr\'{e}chet differentiable and
\ben
F'(h;\Sigma)=v:= -(\mathcal{S}_k \mathrm{S}_k^{-1} Z)'u^i=v_1  + v_2  + v_3,
\enn
where
\ben
\begin{aligned}
v_1 	&:= \mathcal{S}_k'(\psi;\Sigma, h), \\
v_2 	&:= -\mathcal{S}_k \mathrm{S}_k^{-1}\mathrm{S}_k'(\psi;\Sigma, h), \\
v_3  	&:= \mathcal{S}_k \mathrm{S}_k^{-1}Z'(-u^i;\Sigma, h),
\end{aligned}
\enn
with $\psi = \mathrm{S}_k^{-1}Z(-u^i)$.

By cosine substitution and the jump relation of the gradient of $\mathcal{S}_k$,
\ben
v_1=v_0\quad\mbox{in}\,\R^2,
\enn
where $v_0$ is given in \eqref{eq-v0}. Moreover, by \eqref{eq-SonT},
\be\label{eq-v1onSig}
v_{1,\pm} =  - Z^{-1}\tilde{h}_i\cdot \nabla u^s _\pm + Z^{-1}\mathrm{S}_k'\quad \mbox{on}\, \Gamma_i\subset \Sigma.
\en

By the continuity of single-layer potential, $w:= v-v_0=v_2+v_3$ is continuous over $\mathbb{R}^2$. Furthermore, $w$ is a radiating solution of the Helmholtz equation in $\R^2\ba\ov{\Sigma}$ with Dirichlet boundary value
\ben
w =     -Z^{-1}\mathrm{S}_k'(\psi;\Sigma, h) +    Z^{-1}Z'(-u^i;\Sigma, h) \quad \hbox{on } \Sigma.
\enn
With the help of \eqref{eq-fre-Z}, we have
\be\label{eq-wonSig}
w = -Z^{-1}\mathrm{S}_k'(\psi;\Sigma,h) - Z^{-1}\tilde{h}\cdot \nabla u^i\quad \hbox{on } \Sigma.
\en
Combining \eqref{eq-v1onSig} and \eqref{eq-wonSig}, it follows that $v$ is a radiating solution of \eqref{eq-Helmholtz4v}.

By the fact $F'(h;\Sigma)=v$, we have
\ben
\|F(\Sigma+h)-F(\Sigma)-v\|_{H^1(D)}=O(\|h\|_{C^3}^2).
\enn
Furthermore, using the trace theorem and the Green's representation \cite{CK2019} of the far field pattern, we deduce that
\ben
\|\mathbf{F}(\Sigma+h)-\mathbf{F}(\Sigma)-v^\infty\|_{L^2(S^1)}=O(\|h\|_{C^3}^2).
\enn
The proof is complete.
\end{proof}
Note that it's necessary to assume that $z_j\in C^3((-1,1),\R^2)$ for all $j$ in order to ensure the existence of the Fr\'{e}chet derivatives. However, the condition $z_j\in C^2((-1,1),\R^2)$ for all $j$ is enough to ensure the well-posedness of the scattering problem and the uniqueness of the inverse problem.

\begin{corollary}
\label{thm-kerF}
The nullspace of $\mathbf{F}'(\cdot;\Sigma)$ is given by
\be
\label{eq-nullF}
N(\mathbf{F}'(\cdot;\Sigma)) = \{
h\in \Pi_{i=1}^n C^3((0,\pi),\R^2)| \nu(z_i(\cos t)) \cdot h(\cos t) = 0, t\in \mathbb{R}, 1\leq i\leq n
\}.
\en
\end{corollary}
\begin{proof}
Assume that $\mathbf{F}'(h;\Sigma)=0$. By the Rellich's lemma \cite{CK2019} and the analyticity of $v$ in \eqref{eq-Helmholtz4v}, we have $v_\pm = 0$. Hence it follows that
\ben
-\nu\cdot Z^{-1}\tilde{h} \frac{\partial u_\pm}{\partial \nu} = 0 \quad \mbox{on }\Sigma.
\enn
By the Holmgreen's uniqueness theorem, $ \frac{\partial u_\pm}{\partial \nu}$ can not vanish on any open subset of $\Sigma$. Then $\nu\cdot \tilde{h}$ must vanish on $\Sigma$.
On the other hand, with the help of \eqref{eq-Helmholtz4v}, $\mathbf{F}'(h;\Sigma)=0$ as long as $\nu\cdot \tilde{h}$. The proof is complete.
\end{proof}

\subsection{Modified Newton method}

The nonempty nullspace of $\mathbf{F}'(\cdot;\Sigma)$ in \eqref{eq-nullF} indicates that the parameterization of cracks should be chosen carefully to escape from $N(\mathbf{F}'(\cdot;\Sigma))$. In our problem setting, the unknown cracks are supposed to be the type
\ben
z_j(s)=\Big(x_j(s),\, y_j(s)\Big),\quad s\in (-1, 1)
\enn
such that $x_j, y_j \in C^2(-1,1)$ and $x_j$ is monotone for $j=1,2,\cdots,n$.
Based on the Fr\'{e}chet derivatives given in Theorem \ref{thm-Frechet}, we introduce a modified Newton method as described in Algorithm \ref{alg-1}.
The Newton method in \cite{Kress95} is a degenerated version of Algorithm ref{alg-1} when the horizontal axis is fixed.

For numerical similations, $h$ is taken from a finite dimensional subspace of $\Pi_{i=1}^n C^3(-1,1)$.
Different to the  Newton method in \cite{Kress95} and the hybrid method in \cite{KressSerra05}, to avoid self-intersecting curves,
we choose
\be\label{eq-h}
h=\left(d_{j0}+d_{j1}s,\,  \sum_{i=0}^p  c_{ij}T_i(s) \right)_{j=1}^n,
\en
where $T_i$ is the Chebyshev polynomial of degree $i$.
Note that $h$ is a linear combination of
\be
\label{eq-hij}
\begin{aligned}
h_{ij}(s) &= (0,T_i(s)), \quad 1\leq j\leq n,\, 0\leq i\leq p,\\
h_{j}^{(i)}(s) &= (s^i,0), \quad  1\leq j\leq n,\, 0\leq i\leq 1.
\end{aligned}
\en
Here $h_{ij}$ and $h_{j}^{(i)}$ is defined on the $j$-th crack and vanish on the other cracks.
It is easy to check that $h_{ij}\not\in N(\mathbf{F}'(\cdot;\Sigma))$ if $x'_j\neq 0$ and $h_j^{(i)}\not\in N(\mathbf{F}'(\cdot;\Sigma))$ if $y_j'\neq 0$.



To solve the linear system of  \eqref{eq-NewtonEq} stably, we consider the iteration with respect to $p$. Precisely, the reconstruction using smaller $p$ will be used as the initial guess for the reconstruction with bigger $p$.


\begin{algorithm}\label{alg-1}
      \caption{Modified Newton method with a single frequency data}
\KwIn{$u^\infty (\hat{x})$ for $\hat{x}\in S^1$; \\
  \qquad\quad initial guess $\Sigma^{(0)}=\{\Gamma_j = \{z_j^{(0)}(s)\}| j=1,2,\cdots, n\}$.}
\KwOut{cracks reconstruction $\Sigma$.}
Compute the far field  $\tilde{u}^\infty$ by solving the direct problem \eqref{eq-scattering} with cracks $\Sigma=\Sigma^{(0)}$\;
Compute the relative error $J_r:=\|\tilde{u}^\infty-u^\infty \|/ \|u^\infty \|$\;
\For{$p=p^{(0)},p^{(0)}+1,\cdots,m_p$}
{
      \Do{$\delta J_r>\epsilon$}{
             Compute the Fr\'{e}chet derivatives $F_v^{ij}$ for vertical displacement $h_{ij}$ in \eqref{eq-hij} for $1\leq j\leq n, 1\leq i\leq p$\;
             Compute the Fr\'{e}chet derivatives $F_h^{ij}$ for horizontal displacement $h_{j}^{(i)}$ in \eqref{eq-hij} for $1\leq j\leq n, 1\leq i\leq 2$\;
             \uIf{$F_h^{ij}\neq 0$ for all $i,j$}{
                   Solve $\sum_{i,j} c_{ij} F_v^{ij}+\sum_{i,j}  d_{ij}F_h^{ij}= u^\infty -\tilde{u}^\infty$ to get $c_{ij},d_{ij}$\;
                   \textbf{Update} the crack $z_j = z_j+\sum_{i} c_{ij} h_{ij} + \sum_{i} d_{ij}h_j^{(i)}$ for all $j$\;
             }
             \Else{
                   Solve $\sum_{i,j} c_{ij} F_v^{ij}= u^\infty -\tilde{u}^\infty$ to get $c_{ij}$\;
                   \textbf{Update} the crack $z_j = z_j+\sum_{i} c_{ij} h_{ij}$ for all $j$\;
             }
         Compute the far field  $\tilde{u}^\infty$ by solving the direct problem \eqref{eq-scattering} with cracks $\Sigma=\{z_j| j=1,2,\cdots,n\}$\;
         Compute the new error $\wi{J_r}=\|\tilde{u}^\infty-u^\infty \|/ \|u^\infty \|$\;
         Update $\delta J_r = |\wi{J_r}-J_r|$, $J_r=\wi{J_r}$\;
      }
}
\end{algorithm}

Generally speaking, the high frequency data may produce reconstructions with high quality. However, this heavily depends on a fine initial guess.
Based on this observation, we introduce a multi-frequency version of Algorithm \ref{alg-1} as in Algorithm \ref{alg-2}. Precisely, the lower frequency data is used to get a rough reconstruction, which is then used as an initial guess for the modified Newton method with the higher frequency data.
Note that using high frequency data to determine the  horizontal axis is pretty unstable. When utilizing higher frequency data, we always fix the horizontal axis obtained by the reconstruction with low frequency data.

\begin{algorithm}\label{alg-2}
      \caption{Modified Newton method with multi-frequency data.}
  \KwIn{$u^\infty (\hat{x};k)$ for $\hx\in S^1$, $k=k_1,k_2, \cdots, k_m$; \\
  \qquad\quad initial guess $\Sigma^{(0)}=\{\Gamma_j = \{z_j^{(0)}(s)\}| j=1,2,\cdots, n\}$.}
  \KwOut{cracks reconstruction $\Sigma$.}
      Use low frequency data $u^\infty (\cdot;k_1)$, initial guess $\Sigma^{(0)}$ and $p^{(0)}=1$ to run Algorithm \ref{alg-1} until $J_r(k_1)<\epsilon_{k_1}$ to get the reconstruction $\Sigma^{(k_1)}$ of order $p_{k_1}$\;
        Set $p=p_{k_1}$, $\Sigma =\Sigma^{(k_1)}$ \;
  \For{$k=k_2, \cdots, k_m$}
  {
  		Take $u^\infty (\cdot;k)$ as data, $\Sigma$ as initial guess and $p^{(0)}=p$. Run Algorithm \ref{alg-1}, without updating horizontal axis, to get the reconstruction $\Sigma^{(k)}$ of order $p_{k}$ such that $J_r(k)<\epsilon_k$\;
      Set $p=p_{k}$, $\Sigma =\Sigma^{(k)}$ \;
}
\end{algorithm}

\section{Numerical experiments}
\setcounter{equation}{0}

This section is devoted to verify the modified Newton methods introduced in the previous section. For the direct problem, we use the Nyst\"{o}m method as presented in  \cite{ChapkoKress93} and \cite{Kress95}.
The far field patterns $u^{\infty}(\hat{x}_j, d,k)$ are taken evenly at the $32$ observation directions
\ben
\hat{x}_j = \left(\cos \frac{j\pi}{16},\, \sin \frac{j\pi}{16}\right), \quad j=1,2,\cdots,32.
\enn
If not additionally mentioned, we take $d=(1,0)$ and $k=3$.
These data are then stored in the vector $F_\Sigma \in \C^{1 \times 32}$.
We further perturb
$F_\Sigma$ by random noise using
\ben
F_{\Sigma}^{\delta}\ =\ F_{\Sigma} +\delta\|F_{\Sigma}\|\frac{R}{\|R\|},
\enn
where $R$ is a $1 \times 32$ vector containing pseudo-random values
drawn from a normal distribution with mean zero and standard deviation one. The
value of $\delta$ used in our code is $\delta:=\|F_{\Sigma}^{\delta} -F_{\Sigma}\|/\|F_{\Sigma}\|$ and so presents the relative error.
%
In addition, the stopping criterion $\epsilon$ is set to be $0.001$ in all examples.

We have presented two examples in the next two subsections to validate our algorithms.
The Newton-type iterative methods give quite good reconstructions with the price of a fine initial guess. To solve this problem, we show, in the subsection 4.3, the potential of the direct sampling method and the level set method for such a fine initial guess.

\subsection{Example (a): Identifying cracks}
Our first example takes three cracks in the following and is intended to illustrate two advantages of our method. First, the method works for multiple cracks. Second, the initial guess is allowed to be shifted from the true cracks.
\ben
\begin{aligned}
\Gamma_1 &= \{(0.5s+1,\,  0.5\cos(0.5\pi s)+0.2 \sin(0.5\pi s)-0.1\cos(1.5\pi s))| s\in(-1,1)\}, \\
\Gamma_2 &= \{(0.5s-1,\, -1-0.4 \sin(0.5\pi s)+0.1\cos(1.5\pi s))| s\in(-1,1)\}, \\
\Gamma_3 &= \{(s,\, 3+0.3 \cos(0.5\pi s)+0.2\sin(0.5\pi s))| s\in(-1,1)\}.
\end{aligned}
\enn
To avoid the "inverse crime", the cracks are chosen to be trigonometric functions, which do not belong to the polynomial space.

In Figure \ref{fig-a-0noise} we show results of the reconstruction together with the initial data:
\ben
\begin{aligned}
\Gamma_1^{(0)} &= \{(0.5s+1.5, 0)| s\in(-1,1)\}, \\
\Gamma_2^{(0)} &= \{(0.5s-1.5, -1)| s\in(-1,1)\}, \\
\Gamma_3^{(0)} &= \{(s , 3)| s\in(-1,1)\}.
\end{aligned}
\enn
The initial guess has a horizontal translation compared to the true cracks, and thus the classical Newton method in \cite{Kress95} fails to handle.

Figure \ref{fig-a-noise} shows the results with noise at different noise levels. Considering the noise levels up to $0.1$, the reconstructions are quite good.
Table \ref{tab-a} compares the reconstructed coefficients of $\Gamma_1$ with different problem settings. We observe that the data with incident direction $(0,1)$ offers slightly better reconstructions than the data with incident direction $(1,0)$, which may due to the fact that $\Gamma_1$ locates at the right bottom part.

\begin{figure}[h!]
    \centering
    \begin{tabular}{cc}
        \subfigure[Initial guess.]{
            \label{fig-(a)-direct}
            \includegraphics[width=.45\textwidth]{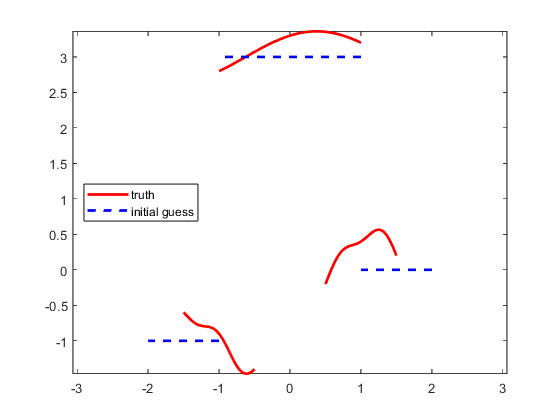}
        } \hspace{0em} &
        \subfigure[Results with $m_p = 4$.]{
            \label{fig-(a)-result}
            \includegraphics[width=.45\textwidth]{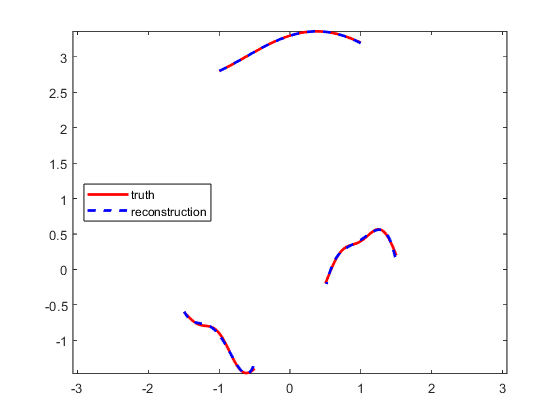}
        }
    \end{tabular}
    \caption{Example (a): initial guess and result without noise.}
    \label{fig-a-0noise}
\end{figure}

\begin{figure}[h!]
    \centering
    \begin{tabular}{cc}
        \subfigure[$\delta = 0.01$, $m_p=4$.]{
            \label{fig-(a)-001noise}
            \includegraphics[width=.45\textwidth]{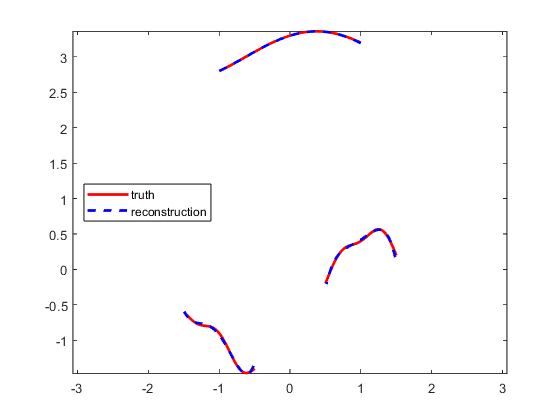}
        } \hspace{0em} &
        \subfigure[$\delta = 0.03$, $m_p = 4$.]{
            \label{fig-(a)-003noise}
            \includegraphics[width=.45\textwidth]{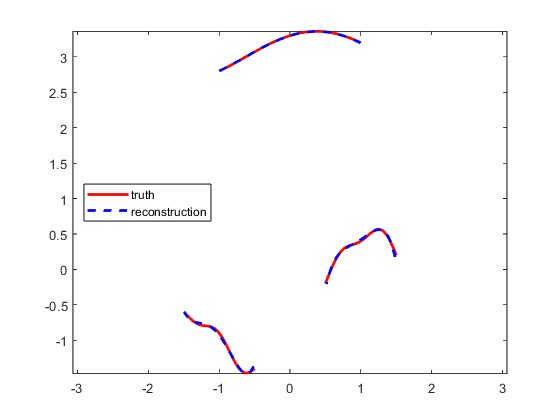}
        }\\
        \subfigure[$\delta = 0.05$, $m_p=4$.]{
            \label{fig-(a)-005noise}
            \includegraphics[width=.45\textwidth]{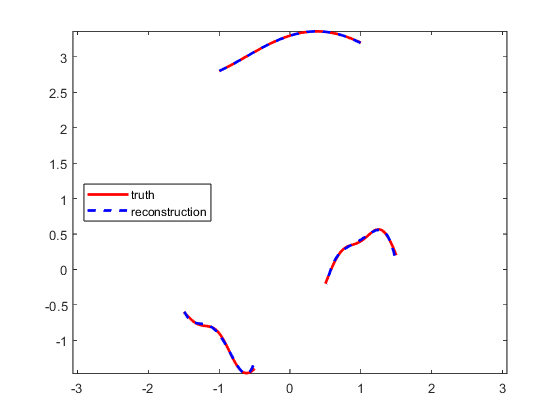}
        } \hspace{0em} &
        \subfigure[$\delta = 0.1$, $m_p = 4$.]{
            \label{fig-(a)-01noise}
            \includegraphics[width=.45\textwidth]{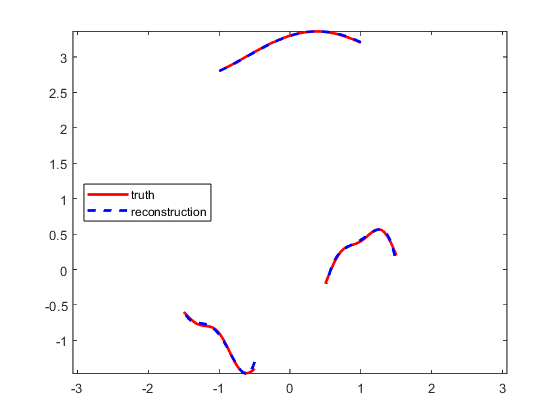}
        }
    \end{tabular}
    \caption{Example (a): results with noise at different noise level $\delta$.}
    \label{fig-a-noise}
\end{figure}

\begin{table}
     \centering
     \begin{tabular}{cccccccccc}
     \hline
     {$u^i$} & 	$\delta$ & $m_p$	&	$d_{1,0}$	& $d_{1,1}$	& $c_{1,0}$	& $c_{1,1}$	& $c_{1,2}$	& $c_{1,3}$ & $c_{1,4}$\\ \hline

     -&-&-& 1&0.5	& 0.263	& 0.227	& -0.220 &-0.028 	&-0.060
     \\\hline
     (1,0)	&	0	& 4	&	1.000 	& 0.486 & 0.273 & 0.219 & -0.219	 & -0.032 & -	0.073		\\
          	&	0.01	& 4	&	1.001 	& 0.463 & 0.274 & 0.219 	& -0.218	 & -0.031 & -0.073	\\
          	&	0.1	& 4	&	1.016	& 0.493 & 0.280 & 0.212 	& -0.224	 & -0.034 & -0.070		\\

     (0,1)	&	0	& 4	&	1.000	& 0.493 & 0.272 & 0.226 & -0.214	 & -0.027 & -	0.061		\\
          	&	0.01	& 4	&	1.000 	& 0.492 & 0.272 & 0.225 	& -0.213	 & -0.025 & -0.063	\\
          	&	0.1	& 4	&	0.995 	& 0.496 & 0.266 & 0.235 	& -0.220	 & -0.034 & -0.048		\\

          \hline
     \end{tabular}
     \caption{Example (a): numerical reconstructions of the $\Gamma_1$. The first row shows the best approximation of $\Gamma_1$ in polynomials space of order $4$.}
  \label{tab-a}
\end{table}

\subsection{Example (b): Identifying complicated curve with two frequencies data}
For the second example, we select a more complex crack
\be
\label{eq-nume-cracks-b}
\Gamma = \{(s, 0.5\cos(\pi s)+0.2 \sin(\pi s)-0.1\cos(3\pi s) + 0.1\sin(5\pi s)), s\in(-1,1)\}
\en
in order to show how to utilize multi-frequency data and get a better reconstruction.

In this example, we use far field patterns at frequencies $k=3, 9$. For the higher frequency $k=9$, the wavelength is $\lambda=2\pi/9$ and the crack is about 3 wavelengths in width. In our simulation, $\epsilon(3)=0.001$ and $\epsilon(9)=\max\{0.01, \delta/2\}$, where $\delta$ is the noise level.

In Figure \ref{fig-b-initial}, we display the initial guess and the reconstruction from data at a single frequency.
Figure \ref{fig-b-only3-10} shows the reconstruction with far field patterns at the lower frequency $k=3$, where $p=10$ is taken automatically by the stopping criteria.  Despite the bad initial guess, it captures the location well and captures roughly the shape.
Using higher frequency $k=9$ can not get a result automatically by the iterative process. Hence we fix $p=5$ and show the results in Figure \ref{fig-b-only9-5}, which is much worse than the results of $k=3$ in Figure \ref{fig-b-only3-5}. This may due to the fact that the reconstruction with higher frequency data depends heavily on a good initial guess.


In Figure \ref{fig-b-noise}, we utilize Algorithm \ref{alg-2} to take advantage of measurement at both frequencies $k=3,9$.
As expected, compared to Figures \ref{fig-b-only3-5} and \ref{fig-b-only3-10}, the quality of reconstruction is greatly improved.
Figures \ref{fig-b-noise001}-\ref{fig-b-noise005} show the results with different noise levels.
Meanwhile, we observe that the quality of the reconstruction deteriorates along the incident direction. Physically, the information from the "shadow region" is very weak, especially for high frequency waves. Obviously, such a defect is aggravated with the increase of the measurement noises.

Figures \ref{fig-b-01}-\ref{fig-b--10} show the reconstructions, respectively, with four different incident directions. This further verify the fact that the quality of the reconstruction deteriorates along the incident direction. To solve this problem, we use the measurements at two opposite incident directions. Figures \ref{fig-b-+-10}-\ref{fig-b-0+-1} verify the effectiveness of our technique.

\begin{figure}[h!]
    \centering
    \begin{tabular}{cc}
        \subfigure[Initial guess.]{
            \label{fig-b-initial}
            \includegraphics[width=.45\textwidth]{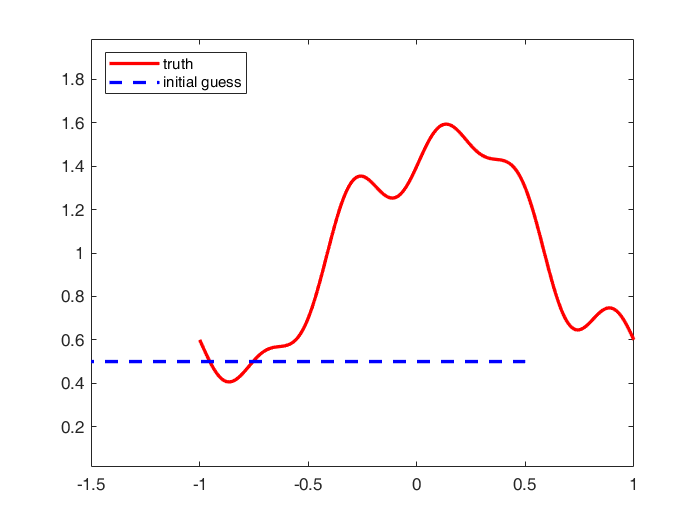}
        } \hspace{0em} &
        \subfigure[Final results of $p=10$ with data at only $k=3$.]{
            \label{fig-b-only3-10}
            \includegraphics[width=.45\textwidth]{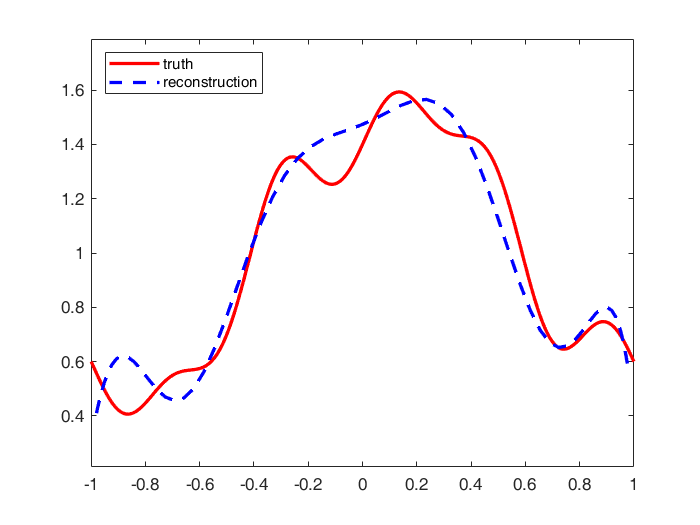}
        }\\
        \subfigure[Results of $p=5$ with data at only $k=3$.]{
            \label{fig-b-only3-5}
            \includegraphics[width=.45\textwidth]{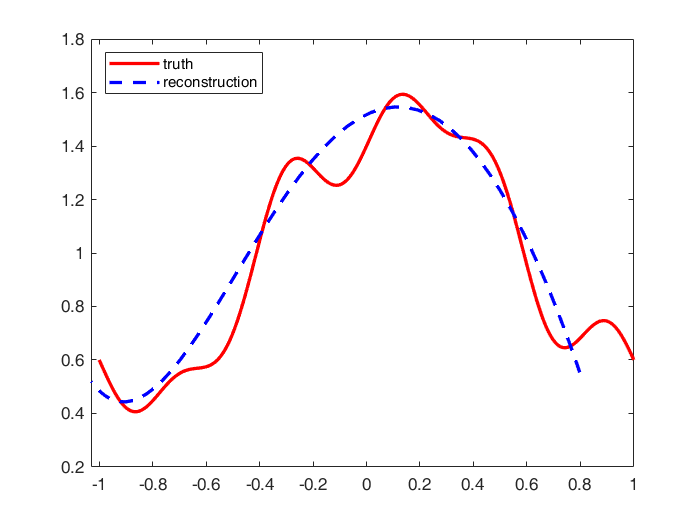}
        } \hspace{0em} &
        \subfigure[Results of $p=5$ with data at only $k=9$.]{
            \label{fig-b-only9-5}
            \includegraphics[width=.45\textwidth]{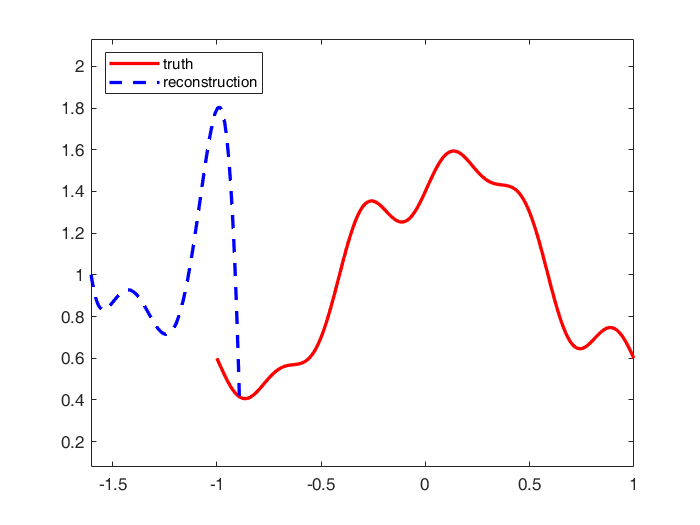}
        }
    \end{tabular}
    \caption{Example (b): the first row are the initial guess and the results with data at $k=3$. The second row compares reconstructions with low and high frequencies. Noise level $\delta=0.01$.}
    \label{fig-b-initial}
\end{figure}

\begin{figure}[h!]
    \centering
    \begin{tabular}{cc}
        \subfigure[$\delta=0$, $p=17$. ]{
            \label{fig-b-noise0}
            \includegraphics[width=.45\textwidth]{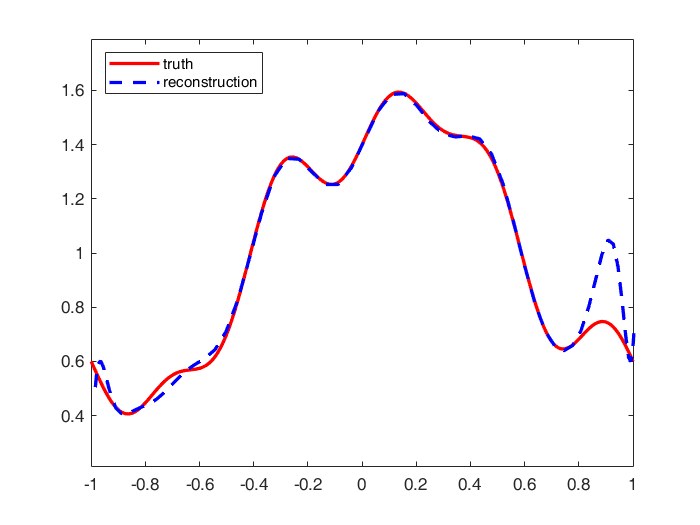}
        } \hspace{0em} &
        \subfigure[$\delta=0.01$, $p=16$.]{
            \label{fig-b-noise001}
            \includegraphics[width=.45\textwidth]{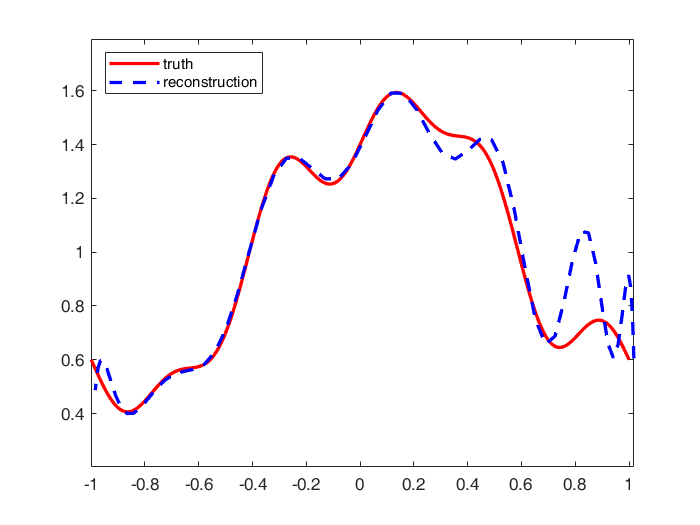}
        }\\
        \subfigure[$\delta=0.03$, $p=16$. ]{
            \label{fig-b-noise003}
            \includegraphics[width=.45\textwidth]{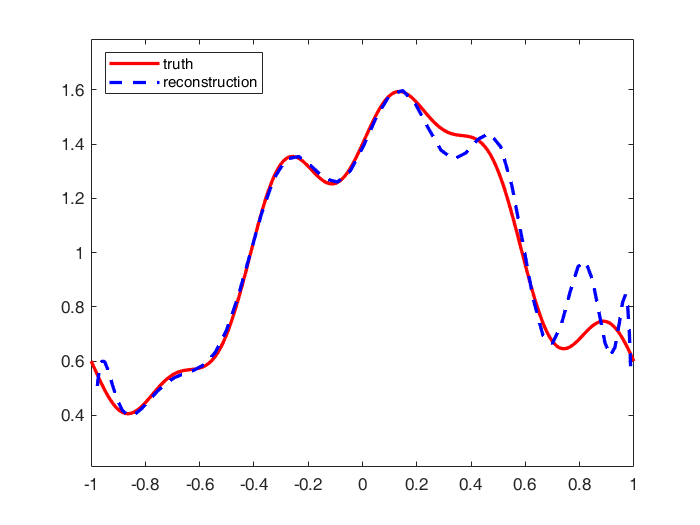}
        } \hspace{0em} &
        \subfigure[$\delta=0.05$, $p=15$.]{
            \label{fig-b-noise005}
            \includegraphics[width=.45\textwidth]{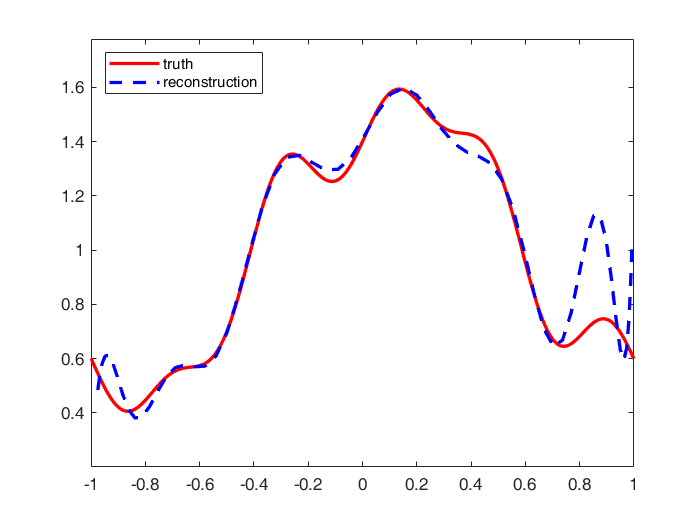}
        }
    \end{tabular}
    \caption{Example (b): results with noise at different noise level $\delta$.}
    \label{fig-b-noise}
\end{figure}

\begin{figure}[h!]
    \centering
    \begin{tabular}{cc}
        \subfigure[Incident direction (1,0), $p=15$. ]{
            \label{fig-b-01}
            \includegraphics[width=.45\textwidth]{figure/b-k9_15_005noise.png}
        } \hspace{0em} &
        \subfigure[Incident direction (0,1), $p=11$. ]{
            \label{fig-b-10}
            \includegraphics[width=.45\textwidth]{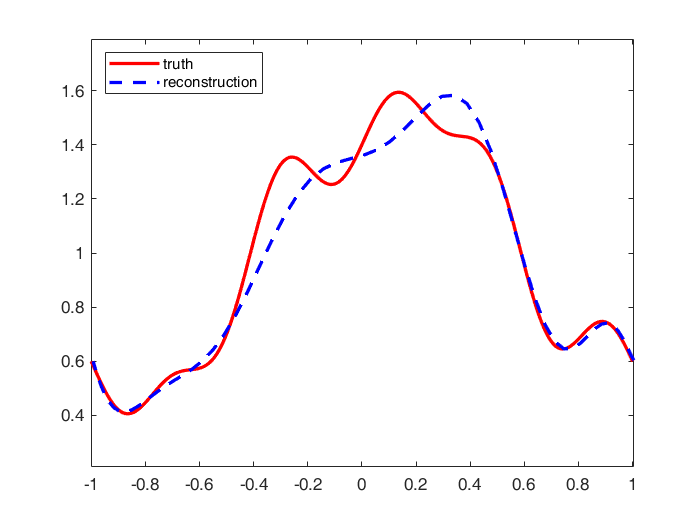}
        }\\
        \subfigure[Incident direction (-1,0), $p=15$. ]{
            \label{fig-b-0-1}
            \includegraphics[width=.45\textwidth]{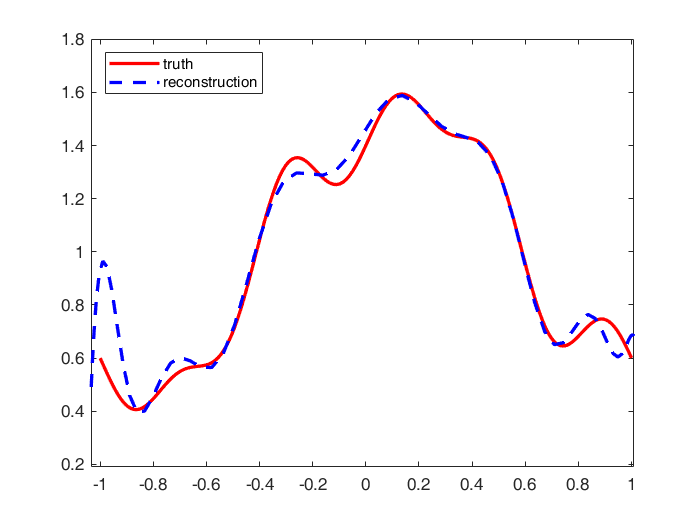}
        } \hspace{0em} &
        \subfigure[Incident direction (0,-1), $p=15$. ]{
            \label{fig-b--10}
            \includegraphics[width=.45\textwidth]{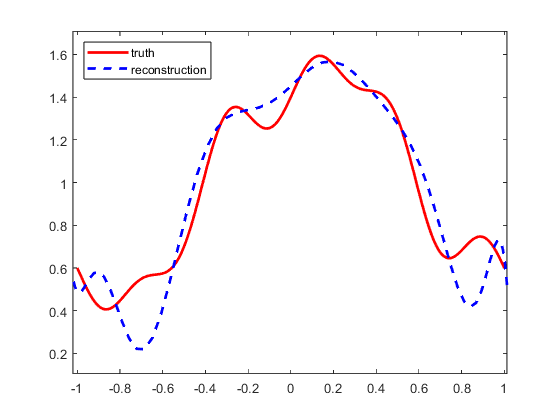}
        }\\
        \subfigure[Incident direction $(\pm 1,0)$, $p=17$. ]{
            \label{fig-b-+-10}
            \includegraphics[width=.45\textwidth]{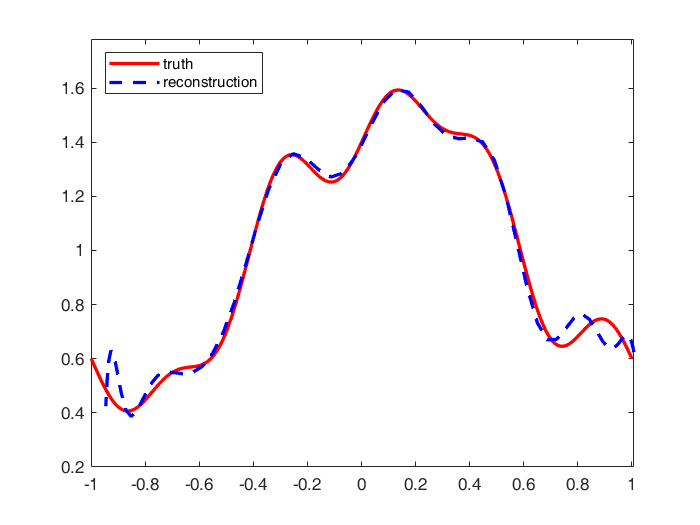}
        } \hspace{0em} &
        \subfigure[Incident direction $(0,\pm 1)$, $p=17$. ]{
            \label{fig-b-0+-1}
            \includegraphics[width=.45\textwidth]{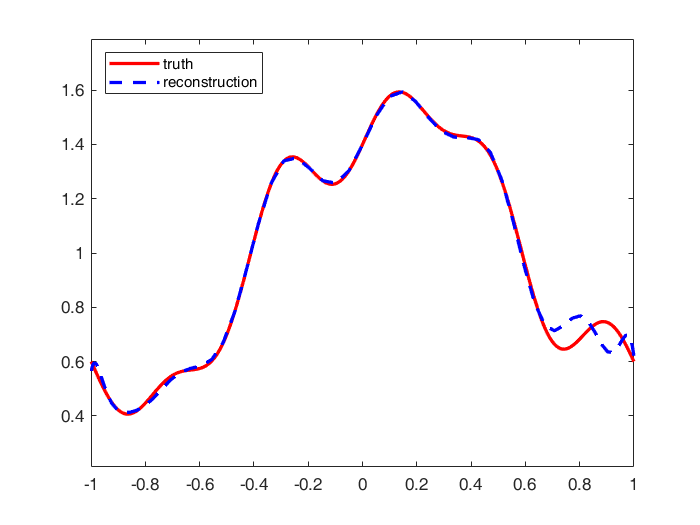}
        }
    \end{tabular}
    \caption{Example (b): The first two rows are results with data from different incident directions. The last row are results using data at two opposite directions. Noise level $\delta=0.05$.}
    \label{fig-b-more}
\end{figure}

\subsection{Initial guess without a prior knowledge}
As observed in the numerical simulations, the initial guess plays a key role in the reconstructions if the Newton type method is applied.
In this subsection, we employ the direct sampling method \cite{LiZou13} and the accelerated level-set method \cite{AudibertHL22}, which are expected to offer a fine initial guess for the modified Newton's method.
We do not use the level-set method of cracks in \cite{AlvaDIM09} because it is numerically unstable in our setting. It seems that cracks scattering is significantly different from the EIT problems discussed in \cite{AlvaDIM09}.

The direct sampling method \cite{LiZou13} is to calculate the indicator function
\be
I(z) := \left|
\sum_{j=1}^N u^\infty (\hat{x}_j)e^{ik\hat{x}\cdot z}
\right|,\quad z\in \mathbb{R}^2.
\en
In Figure \ref{fig-DSM}, we utilize this method for the two examples we studied in the previous subsections.

\begin{figure}[h!]
    \centering
    \begin{tabular}{cc}
        \subfigure[Example (a)]{
            \includegraphics[width=.45\textwidth]{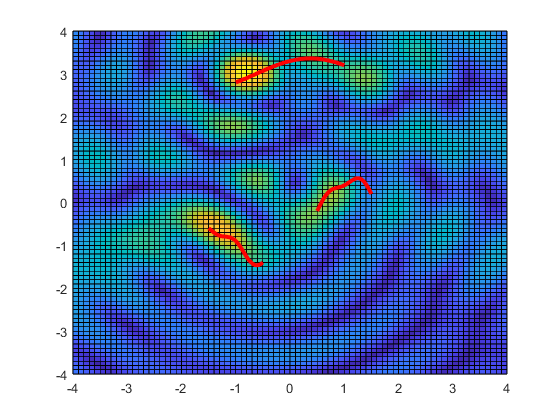}
        }
        \subfigure[Example (b)]{
            \includegraphics[width=.45\textwidth]{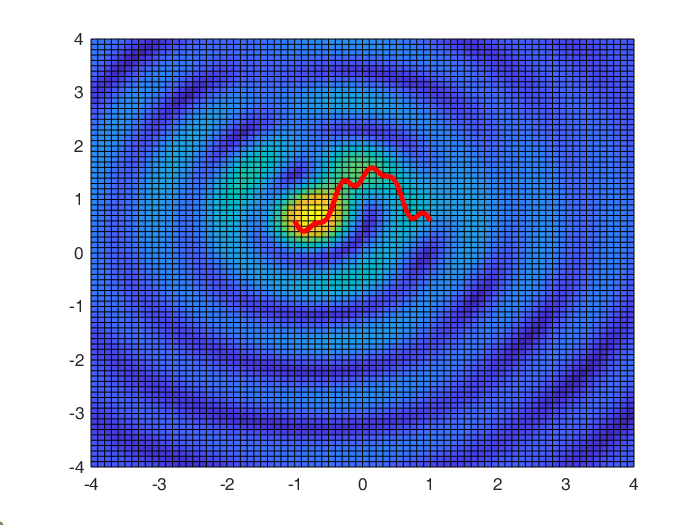}
        }
    \end{tabular}
    \caption{Direct sampling method: red lines are the true cracks.  Frequency $k=3$, noise level $\delta=0.1$.}
    \label{fig-DSM}
\end{figure}

The level-set method represents the obstacle by a level-set function $\phi$ by
\ben
D_\phi = \{x|\phi(x)<0\}.
\enn
Then $\phi$ is updated (from the knowledge of the residual error) until the cost function $J$ converges. The cost function is defined as
\ben
J(D_\phi) := \frac{1}{2N} \sum_{j=1}^N \left|u^\infty (\hat{x}_j)-\tilde{u}^\infty(\hat{x}_j)\right|^2,
\enn
where $\tilde{u}^\infty(x_j)$ is the far field pattern of the scattered field by $D_\phi$. It shows that
\ben
J'(D)(\theta) = \int_{\partial D} \frac{\partial u(y)}{\partial \nu}\frac{\partial w(y)}{\partial \nu}(\theta\cdot \nu)ds(y)
\enn
for $\theta \in C^1(\partial D, \mathbb{R}^2)$.
Here $u$ is the total field of $D$ from the plane wave $e^{ikx\cdot d}$ and $w$ is the total field of $D$ from incident field
\ben
w^i(x)  = \frac{C}{N} \sum_{j=1}^N \overline{u^\infty _D(x_j)-u^\infty (x_j)}e^{-ik\hat{x}_j\cdot x},
\enn
where $C=\mbox{exp}(i\pi/4)/\sqrt{8\pi k}$.
More details about the accelerated level-set method can be found in \cite{AudibertHL22}. Figure \ref{fig-LS} shows the reconstructions of three cracks by using level-set method.

\begin{figure}[h!]
    \centering
    \begin{tabular}{cc}
        \subfigure[Initial guess.]{
            \includegraphics[width=.45\textwidth]{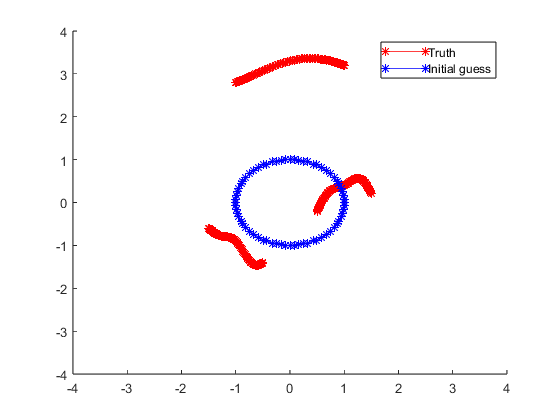}
        } \hspace{0em} &
        \subfigure[Residual $J$ over iterations.]{
            \includegraphics[width=.45\textwidth]{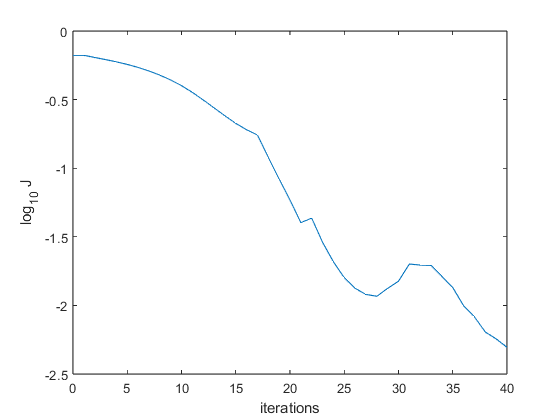}
        }\\
        \subfigure[$20$ iterations.]{
            \includegraphics[width=.45\textwidth]{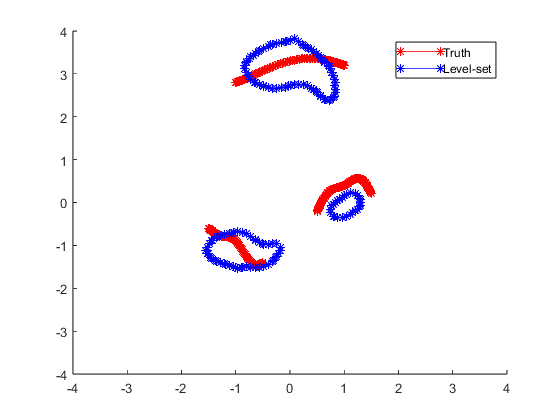}
        } \hspace{0em} &
        \subfigure[$40$ iterations.]{
            \includegraphics[width=.45\textwidth]{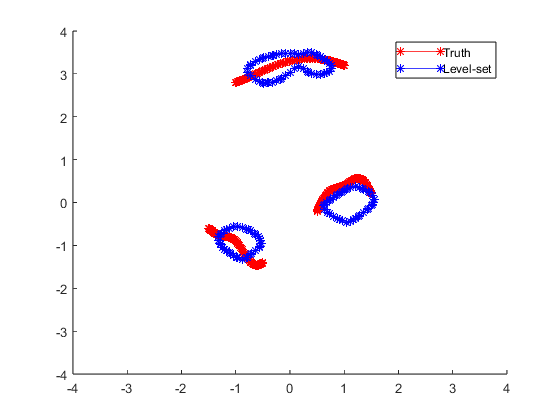}
        }
    \end{tabular}
    \caption{Level-set method: red lines are the true cracks.  Frequency $k=1$.}
    \label{fig-LS}
\end{figure}

In Figure \ref{fig-LS} we use far field data at frequency $k=1$ instead of $k=3$. We compare the results with different frequencies in Figure \ref{fig-LS2}. It shows that data at a higher frequency may miss the crack that is away from the initial guess.

Finally, we note that the validity of these methods has not been proved for our problems setting.
However, numerical experiments demonstrate that both methods offer a reasonable estimate of the actual cracks.

\begin{figure}[h!]
    \centering
    \begin{tabular}{cc}
        \subfigure[$k=1$, 40 iterations.]{
            \includegraphics[width=.45\textwidth]{figure/LS-40.png}
        } \hspace{0em} &
        \subfigure[$k=3$, 40 iterations.]{
            \includegraphics[width=.45\textwidth]{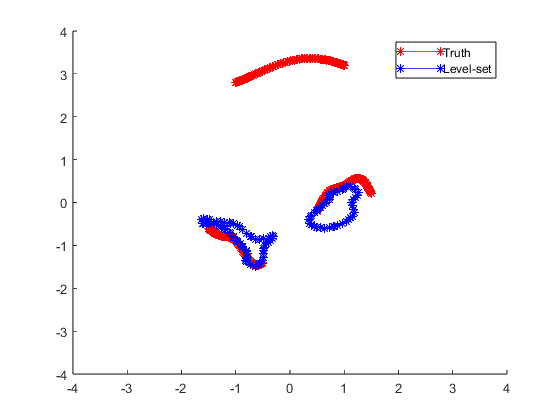}
        }
    \end{tabular}
    \caption{Level-set method wit different frequencies.}
    \label{fig-LS2}
\end{figure}


\section*{Acknowledgment}
The research of X. Liu is supported by the NNSF of China under grant 12371430 and Beijing Natural Science Foundation Z200003.

\setcounter{equation}{0}

\section{Appendix}
This section is devoted to two remarks on the uniqueness of the cracks from the multi-frequency far field patterns with a single incident direction. We show that uniqueness of the impedance cracks does not hold. However, the uniqueness for sound soft cracks still holds if the background is an unknown inhomogeneous medium.

\subsection{Non-uniqueness for the impedance cracks.}

Let $\mathcal{B}u=\frac{\partial u}{\partial \nu}+\lambda u$ on $\Gamma$ be the impedance boundary condition, where $u$ is the total field. We consider the uniqueness of the inverse problem from the far field at one fixed incident direction and all frequencies.

When the incident wave is a plane wave $u^i= e^{ikx\cdot d}$, we have
\ben
\mathcal{B}u^i = \left(\lambda+ik(\nu\cdot d)\right)u^i,
\enn
which will vanish on $\Gamma$ if $\nu\cdot d = -\frac{\lambda}{ik}$.

For sound hard cracks, $\lambda=0$. Let $d = (0,1)$, then $\mathcal{B}u^i$ will vanish on any line cracks with normal $\pm d^\perp$, for example, $\Gamma = \{0\}\times(-1,1)$. Consequently, there is no scattered field, and hence, no uniqueness for the inverse problem.

For general impedance cracks where $\lambda \neq 0$, the uniqueness remains unresolved. This is due to the fact that maximum principle argument in the proof of Theorem \ref{thm-unique} does not work.

\subsection{Uniqueness for sound soft cracks in an unknown inhomogeneous medium}

As stated in \cite{LiuLiu17}, an unknown inhomogeneous medium will not affect the uniqueness of the cracks.
Let $\Sigma$ be the cracks defined in the introduction section, and let $V\in L^\infty \left(\mathbb{R}^2\backslash\overline{\Sigma}\right)$ be an inhomogeneous medium with compact support $\mbox{supp } V\subset \Omega$, where $\Omega$ is a bounded open set. The scattering of incident field $u^i$ is formulated as
\begin{equation}
\label{eq-scattering-V}
    \left\{
\begin{aligned}
    &\Delta u + k^2(1+V) u = 0 &\hbox{in }\mathbb{R}^2\backslash \overline{\Sigma} \\
    &u = 0                 &\hbox{on } \Sigma \\
    & \lim_{r\to\infty} \sqrt{r}\left(\frac{\partial (u-u^i) }{\partial r} - ik (u-u^i) \right)=0        &\hbox{}
\end{aligned}
\right..
\end{equation}

The difference between \eqref{eq-scattering-V} and those in \cite{LiuLiu17} is that cracks has  singularities at the endpoints. Fortunately, thanks to the idea due to Heinz \cite[Theorem 3.9]{CK2019}, the uniqueness of the direct problem still holds. We present the result without proof as follows.
\begin{theorem}
\label{thm-uniqueness4DP-V}
    The Dirichlet problem \eqref{eq-scattering-V} has at most one solution.
\end{theorem}

Define the volume potential
\begin{equation*}
 ( \mathcal{G}_V\phi) (x) := \int_{\Omega\backslash\overline{\Sigma}} \Phi_k(x,y) \phi(y) {\rm d}y, \quad x\in \mathbb{R}^2,
\end{equation*}
which is bounded from $L^2(\Omega\backslash\overline{\Sigma})$ to $H^2_{loc}(\Omega)$.
We also define the restriction of $\mathcal{G}_V$ to the cracks $\Sigma$ by
\begin{equation*}
 ( {G}_V\phi) (x):= \int_{\Omega\backslash\overline{\Sigma}} \Phi_k(x,y) \phi(y) {\rm d}y, \quad x\in \Sigma.
\end{equation*}
Moreover, $G_V:L^2(\Omega\backslash\overline{\Sigma})\to  H^{1}(\Sigma)$ is compact.

Then we can represent the scattered field as the sum of single layer potential and volume potential in $\Omega$.
\begin{theorem}
\label{thm-int-representation-V}
Let $u\in H^1_{loc}(\mathbb{R}^2\backslash\overline{\Sigma})$ be a solution to \eqref{eq-scattering-V} for incident field $u^i$. Then the pair  $(u|_{\Omega\backslash\overline{\Sigma}}, \phi)\in L^2(\Omega\backslash\overline{\Sigma})\times \mathcal{X}$ solves the following integral system
\be
\label{eq-sys-1-V}
u &=& u^i + k^2 \mathcal{G}_V u + \mathcal{S}\left(W-\frac{2\pi}{\ln k}I\right)\phi, \quad \mbox{in } \Omega\backslash\overline{\Sigma},\\
\label{eq-sys-2-V}
0 &=& u^i + k^2 {G}_Vu + \mathrm{S}\left(W-\frac{2\pi}{\ln k}I\right)\phi,\quad \mbox{on } \Sigma.
\en
Furthermore, the system of integral equation \eqref{eq-sys-1-V}-\eqref{eq-sys-2-V} is uniquely solvable.
\end{theorem}

\begin{proof}
First we verify that solution $u\in L^2(\Omega\backslash\overline{\Sigma})$ of \eqref{eq-sys-1-V}-\eqref{eq-sys-2-V} is the solution of \eqref{eq-scattering-V}. Actually we extend $u$ into $\mathbb{R}^2\backslash\overline{\Sigma}$ by the right hand side of \eqref{eq-sys-1-V}. By the mapping property of $\mathcal{G}_V$ and $\mathcal{S}$, $u$ belongs to $H^1_{loc}(\mathbb{R}^2\backslash \overline{\Sigma})\cap H^1(\Sigma)$ and is a radiating solution. Moreover, the boundary condition is satisfied from \eqref{eq-sys-2-V}.

Now we prove that \eqref{eq-sys-1-V}-\eqref{eq-sys-2-V} have exactly one solution $(u,\phi)\in L^2(\Omega\backslash\overline{\Sigma})\times \mathcal{X}$. The system \eqref{eq-sys-1-V}-\eqref{eq-sys-2-V} can be rewritten in a matrix form
\be
\label{eq-sys-matrix}
(\mathscr{A}+\mathscr{B})U = F
\en
with
\ben
\mathscr{A}=
\begin{pmatrix}
-I &0\\
0 &\mathrm{S}_0
\end{pmatrix},\quad
\mathscr{B}=
\begin{pmatrix}
k^2\mathcal{G}_V & \mathcal{S}\\
k^2{G}_V&\mathrm{S}-\mathrm{S}_0
\end{pmatrix},\quad U = \begin{pmatrix}u, \tilde{\phi} \end{pmatrix},\quad F = \begin{pmatrix}-u^i\\-u^i \end{pmatrix},
\enn
where $\tilde{\phi}= (W-\frac{2\pi}{\ln k}I)\phi$. Note that $\phi$ is uniquely determined by $\tilde{\phi}$.

It is known that $\mathscr{A}$ is bounded and has a bounded inverse because Kress \cite{Kress95} prove that $S_0$ has a bounded inverse from $H^1(\Sigma)$ to $L^2(\Sigma)$. Furthermore, $\mathscr{B}$ is compact because $H^1_{loc}(\Omega), H^1_{loc}(\Omega\backslash\overline{\Sigma})$ can be compactly embedded into $L^2(\Omega\backslash\overline{\Sigma})$ and that $\mathrm{S}-\mathrm{S}_0$ is compact from $\mathcal{X}$ to $H^1(\Sigma)$.
Thus $\mathscr{A}+\mathscr{B}$ is a Fredholm operator. To prove the Theorem, it suffice to prove the uniqueness of \eqref{eq-sys-matrix}.

 Let $u^i=0$, then Theorem \ref{thm-uniqueness4DP-V} implies that $u=0$ in $\mathbb{R}^2\backslash \overline{\Sigma}$. Hence $v:=\mathcal{S}\phi=0$ in $\Omega\backslash\overline{\Sigma}$. Now by jump relation in the $L^p$ sense \cite[p.100]{Kress14}
 \ben
 \phi = \frac{\partial v_-}{\partial \nu}-\frac{\partial v_+}{\partial \nu} =0.
 \enn
The proof is finished.
\end{proof}

Note that the volume potential term possesses a $k^2$ factor, which should not change the low frequency expansion \eqref{eq-asym-us}. We prove the expansion in the following by carefully tackling the inhomogeneous term as in \cite{LiuLiu17}.
\begin{theorem}
Let $u^i=e^{ik x\cdot d}$ for some fixed incident direction $d$. Then the solution $u$ of problem \eqref{eq-scattering-V} has the following low frequency behavior
\ben
\begin{aligned}
u   &= u^i  + \mathcal{S} \left(W-\frac{2\pi}{\ln{k}}I\right) \phi +\mathcal{G}_V u\\
        &= - \frac{2\pi}{|\Sigma|\ln{k}} \left[(\mathcal{S}_0 {W} +{L}){A}^{-1}\mathrm{S}_0(1)-\mathcal{S}_0(1)\right]  + O\left(\frac{1}{(\ln k)^{2}}\right),\quad  k\to 0,\quad x\in \mathbb{R}\backslash\overline{\Sigma},
\end{aligned}
\enn
where $A:=\mathrm{S}_0W+L$ and $\mathcal{A}=\mathcal{S}_0W+L$.
\end{theorem}
\begin{proof}
By the asymptotic expansion of the fundamental solution \eqref{eq-asym-Phik}, the left hand side of the \eqref{eq-sys-matrix} has the expansion
\ben
\begin{aligned}
&\left(
\begin{pmatrix}
-I& \mathcal{S}_0\\
0 & \mathrm{S}_0
\end{pmatrix}+\ln k
\begin{pmatrix}
0& -2\pi L\\
0 & -2\pi L
\end{pmatrix}+ C
\begin{pmatrix}
0&  L\\
0 & L
\end{pmatrix}+
O(k^2\ln k)
\right)
\begin{pmatrix}
I& 0\\
0 & W-\frac{2\pi}{\ln k}L
\end{pmatrix}
\begin{pmatrix}
u\\
\phi
\end{pmatrix}\\
&=
\left(
\begin{pmatrix}
-I& \mathcal{S}_0W+L\\
0 & \mathrm{S}_0W+L
\end{pmatrix}-\frac{2\pi}{\ln k}
\begin{pmatrix}
0& \mathcal{S}_0+C L\\
0 & \mathrm{S}_0+C L
\end{pmatrix}+
O(k^2\ln k)
\right)
\begin{pmatrix}
u\\
\phi.
\end{pmatrix}
\end{aligned}
\enn

Note that
\ben
\begin{pmatrix}
-I& \mathcal{S}_0W+L\\
0 & \mathrm{S}_0W+L
\end{pmatrix}^{-1} =
\begin{pmatrix}
-I& \mathcal{A}A^{-1}\\
0 & A^{-1}
\end{pmatrix}.
\enn
By the Neumann series, we have
\ben
\begin{aligned}
&
\left(
\begin{pmatrix}
-I& \mathcal{S}_0W+L\\
0 & \mathrm{S}_0W+L
\end{pmatrix}-\frac{2\pi}{\ln k}
\begin{pmatrix}
0& \mathcal{S}_0+C L\\
0 & \mathrm{S}_0+C L
\end{pmatrix}+
O(k^2\ln k)
\right)^{-1}\\
&=\left(I-\frac{2\pi}{\ln k}
\begin{pmatrix}
0& \mathcal{A}A^{-1}(\mathrm{S}_0+CL)-\mathcal{S}_0+C L\\
0 & A^{-1}(\mathrm{S}_0+C L)
\end{pmatrix}+
O(k^2\ln k)\right)^{-1}
\begin{pmatrix}
-I& \mathcal{A}A^{-1}\\
0 & A^{-1}
\end{pmatrix}\\
&=\left(
I+ \frac{2\pi}{\ln k}
\begin{pmatrix}
0& \mathcal{A}A^{-1}(\mathrm{S}_0+CL)-\mathcal{S}_0+C L\\
0 & A^{-1}(\mathrm{S}_0+C L)
\end{pmatrix}+
O\left(\frac{1}{\ln^2k}\right)
\right)
\begin{pmatrix}
-I& \mathcal{A}A^{-1}\\
0 & A^{-1}
\end{pmatrix}.
\end{aligned}
\enn
For simplicity, we abuse the notation $I$ as identical mapping on any function space.

Moreover, the plane wave admits the expansion $u^i(x) = 1+kx\cdot d +O(k^2)$ on $\Sigma$. Then the results follows by straightforward calculation.
\end{proof}

Comparing this with \eqref{eq-asym-us} of Theorem \ref{u_lowfrequency}, we observe that the $\frac{1}{\ln k}$ terms are exactly the same, and therefore we have the following uniqueness result even the background is an unknown inhomogeneous medium.
\begin{theorem}
\label{thm-unique-V}
Assume $(\Sigma,V)$ and $(\hat{\Sigma},\hat{V})$ are two collections of cracks and medium which produce same far field for incident field patterns $u^i = e^{ik x\cdot d}$ with $k = \{k_j|j = 1,2,\cdots, k_j \to 0\}$ and fixed $d$. Then $\Sigma=\hat{\Sigma}$.
\end{theorem}

\bibliographystyle{SIAM}

\begin{thebibliography}{99}
\bibitem{AlessanRondi05} G. Alessandrini and L. Rondi, \textit{Determining a sound-soft polyhedral scatterer by a single far-field measurement}, P. Am. Math. Soc., 133 (2005), pp. 1685-1691.

\bibitem{AlvaDIM09} D. \'{A}lvarez, O. Dorn, N. Irishina, and M. Moscoso, \textit{Crack reconstruction using a level-set strategy}, J. Comput. Phys., 228 (2009), pp. 5710-5721.

\bibitem{AmmariKLP10} H. Ammari, H. Kang, H. Lee, and W.-K. Park, \textit{Asymptotic imaging of perfectly conducting cracks}, SIAM J. Sci. Comput., 32 (2010), pp. 894-922.

\bibitem{AmmariGKP11} H. Ammari, J. Garnier, H. Kang, W.-K. Park, and K. S\o lna, \textit{Imaging schemes for perfectly conducting cracks}, SIAM J. Appl. Math., 71(2011), pp. 68-91.

\bibitem{AudibertHL22} L. Audibert, H. Haddar, and X. Liu, \textit{An accelerated level-set method for inverse scattering problems}, SIAM J. Imaging Sci., 15 (2022), pp. 1576-1600.


\bibitem{BaoTriki13} G. Bao and F. Triki, \textit{Reconstruction of a defect in an open waveguide}, Sci. China Math., 56 (2013), pp. 2539-2548.

\bibitem{BaoLRX15} G. Bao, S. Lu, W. Rundell, and B.Xu, \textit{A recursive algorithm for multiFrequency acoustic inverse source problems}, SIAM J. Numer. Anal., 53 (2015), pp. 1608-1628.


\bibitem{BaoTriki10} G. Bao and F. Triki, \textit{Error estimates for the recursive linearization of inverse medium problems}, J. Comput. Math., 28 (2010), pp. 725-744.


\bibitem{CakoniColton03} F. Cakoni and D. Colton, \textit{The linear sampling method for cracks}, Inverse Probl., 19(2003), pp. 279-295.

\bibitem{CacoColtMonk01}F. Cakoni, D.  Colton and  P. Monk, \textit{The direct and inverse scattering problems for partially coated obstacles}. Inverse Probl., 17(2001), pp. 1997-2015.

\bibitem{ChapkoKress93}R. Chapko and R. Kress, \textit{On a quadrature method for a logarithmic integral equation of the first kind},  Contributions in Numerical Mathematics, 1993, pp. 127-140.

\bibitem{CK2019} D. Colton and R. Kress, \textit{Inverse Acoustic and Electromagnetic Scattering Theory}, 4th ed.,  Springer, Berlin, 2019.

\bibitem{GaoDM18} P. Gao, H. Dong and F. Ma, \textit{Inverse scattering via nonlinear integral equations method for a sound-soft crack with phaseless data}, Appl. Math., 63 (2018), pp. 149-165.

\bibitem{GuoWuYan18} J. Guo, Q. Wu and  G. Yan, \textit{The factorization method for cracks in elastic scattering}, Inverse Probl. Imag., 12 (2018), pp. 349-371.

\bibitem{HauptmannIIS19} A. Hauptmann, M. Ikehata, H. Itou and  S. Siltanen,  \textit{Revealing cracks inside conductive bodies by electric surface measurements}, Inverse Probl., 35 (2019), p. 025004.

\bibitem{HuLiu14} G. Hu and X. Liu, \textit{Unique determination of balls and polyhedral scatterers with a single point source wave}, Inverse Probl., 30 (2014), p. 065010.

\bibitem{Isakov90} V. Isakov, \textit{On uniqueness in the inverse transmission scattering problem}, Comm. Part. Diff. Equa., 15 (1990), pp. 1565-1587.

\bibitem{KirschRitter00} A. Kirsch and S. Ritter, \textit{A linear sampling method for inverse scattering from an open arc}, Inverse Probl.,16 (2000), pp. 89-105.

\bibitem{Kress95} R. Kress, \textit{Inverse scattering from an open arc}, Math. Method Appl. Sci., 18(1995), pp. 267-293.

\bibitem{Kress96} R. Kress, \textit{Inverse elastic scattering from a crack}, Inverse Probl., 12(1996), pp. 667-684.

\bibitem{Kress99} R. Kress,  \textit{On the low wave number behavior of two-dimensional scattering problems for an open arc}. Zeitschrift F\'{l}\'{z}r Analysis Und Ihre Anwendungen, 18 (1999), pp. 297-305.

\bibitem{Kress14} R. Kress, \textit{Linear Integral Equations}, vol. 82. in Applied Mathematical Sciences, Springer, New York, 2014.

\bibitem{KressBro87} R. Kress and B. Brosowski, \textit{On the low wave number asymtotics for the two-dimensional exterior dirichlet problem for the reduced wave equation}, Math. Meth. Appl. Sci., 9 (1987), pp. 335-341.

\bibitem{KressLee03} R. Kress and K.-M. Lee, \textit{Integral equation methods for scattering from an impedance crack}, J.  Comput. Appl. Math., 161(2003), pp. 161-177.

\bibitem{KruLiuSini08} P. Krutitskii, J. Liu and M. Sini,\textit{Reconstruction of complex cracks by far field measurements}, 6th Int. Confer. Inverse Problems in Engineering: Theory and Practice, J. Physics: Conference Series, 135, (2008).

\bibitem{KressSerra05} R. Kress and P. Serranho, \textit{A hybrid method for two-dimensional crack reconstruction}, Inverse Probl., 21 (2005), pp. 773-784.

\bibitem{KressVinto08} R. Kress and N. Vintonyak, \textit{Iterative methods for planar crack reconstruction in semi-infinite domains}, J. Inverse Ill-Pose. P., 16 (2008), pp. 743-761.

\bibitem{Lee10} K.-M. Lee, \textit{A two step method in inverse scattering problem for a crack}, J. Math. Phys., 51(2010), p. 023529.

\bibitem{LiLiu23} J. Li and X. Liu, \textit{Reconstruction of multiscale electromagnetic sources from multifrequency electric far field patterns at sparse observation directions}, Multiscale Model. Sim., 21 (2023), pp. 753-775.

\bibitem{LiLiuShi23} J. Li,X. Liu and Q. Shi, \textit{Reconstruction of multiscale elastic sources from multifrequency sparse far field patterns}, SIAM J. Appl. Math.,83(2023), pp. 1915-1934.

\bibitem{LiZou13} J. Li and  J. Zou, \textit{A direct sampling method for inverse scattering using far-field data}, Inverse Probl. Imag.,7(2013), pp. 757-775.

\bibitem{LiuLiu17}  H. Liu and X. Liu, \textit{Recovery of an embedded obstacle and its surrounding medium from formally determined scattering data}, Inverse Probl., 33(2017), p. 065001.

\bibitem{LiuZou06} H. Liu and J. Zou, \textit{Uniqueness in an inverse acoustic obstacle scattering problem for both sound-hard and soundsoft polyhedral scatterers}, Inverse Probl., 22 (2006), pp. 515-524.

\bibitem{Monch97} L. M\"{o}nch, \textit{On the inverse acoustic scattering problem by an open arc: the sound-hard case}, Inverse Probl., 13 (1997), pp. 1379-1392.

\bibitem{NakaUhlmWang03}G. Nakamura, G. Uhlmann and J.-N. Wang, \textit{Reconstruction of cracks in an inhomogeneous anisotropic elastic medium}. Journal de Math\'{e}matiques Pures et Appliqu\'{e}es, 82 (2003), pp. 1251-1276.

\bibitem{Park14} W.-K. Park, \textit{Multi-frequency subspace migration for imaging of perfectly conducting, arc-like cracks in full- and limited-view inverse scattering problems}, J. Comput. Phys., 283 (2015), pp. 52-80.

\bibitem{Park18} W.-K. Park, \textit{Direct sampling method for retrieving small perfectly conducting cracks}, J. Comput. Phys.,373 (2018), pp. 648-661.
.

\bibitem{Valtul04} A. O. Vatul'yan, \textit{The determination of the configuration of a crack in an anisotropic medium}, J. Appl. Math. Mec., 68 (2004), pp. 163-170.

\bibitem{XuMH24} X. Xu, G. Ma and G. Hu, \textit{Detection of a piecewise linear crack with one incident wave}, arXiv preprint (2024), arXiv:2405.05179.

\bibitem{Yan08}G. Yan, \textit{The problem of scattering by a partially coated crack}, Nonlinear Anal.-Theor., 68 (2008), pp. 932-939.



\end{thebibliography}

\end{document}